\documentclass[11pt]{amsart}
\usepackage{latexsym}
\usepackage{amsmath,amsthm,amsfonts,amssymb,graphicx,epsfig,latexsym,color,mathrsfs}
\usepackage{epsf,enumerate}

\newtheorem{thm}{Theorem}[section]
\newtheorem{lemma}[thm]{Lemma}
\newtheorem{cor}[thm]{Corollary}
\newtheorem{prop}[thm]{Proposition}
\newtheorem{defn}[thm]{Definition}
{}

\theoremstyle{remark}

\newtheorem{discussion}[thm]{Discussion}

\newtheorem*{definition*}{Definition}
\newtheorem*{remark*}{Remark}

\def\R{{\mathcal R}}

\def\P{{\mathcal P}}
\def\D{{\mathcal D}}

\def\F{{\mathcal F}}

\renewcommand{\phi}{\varphi}

\newcommand{\cvn}{\text{cv}_N}

\newcommand{\FN}{F_N}

\newcommand{\ssm}{\smallsetminus} 

\newcommand{\smf}{\smallfrown}

\begin{document}

\title{Reducing Systems for Very Small Trees}

\author{Patrick Reynolds}

\address{\tt Department of Mathematics, University of Utah, 155 S 1400 E, Room 233, Salt Lake City, Utah 84112, USA}
  \email{\tt
  reynolds@math.utah.edu} 

\date{\today}

\begin{abstract}
We study very small trees from the point of view of reducing systems of free factors, which are analogues of reducing systems of curves for a surface lamination; a non-trivial, proper free factor $F \leq \FN$ \emph{reduces} $T$ if and only if $F$ acts on some subtree of $T$ with dense orbits.  We characterize those trees, called \emph{arational}, which do not admit a reduction by any free factor: $T$ is arational if and only if either $T$ is free and indecomposable or $T$ is dual to a surface with one boundary component equipped with an arational measured foliation.  To complement this result, we establish some results giving control over the collection of all factors reducing a given tree.  As an application, we deduce a form of the celebrated Bestvina-Handel classification theorem for elements of $Out(\FN)$ \cite{BH92}.  We also include an appendix containing examples of very small trees.  The results of this paper are used in \cite{BR12}, where we describe the Gromov boundary of the complex of free factors.
\end{abstract}

\maketitle

\section{Introduction}

This article is about the structure of very small $\FN$-trees that are not free and simplicial; such trees represent points in the boundary of the Culler-Vogtmann Outer space \cite{CL95}, and every very small $\FN$-tree arises in this way \cite{BF94}.  We use $\partial \cvn$ do denote the set of very small $\FN$-trees that are not free and simplicial.

For $T \in \partial \cvn$, say that a non-trivial, proper free factor $F \leq \FN$ \emph{reduces} $T$ if there is an $F$-invaraint subtree $Y \subseteq T$ such that the action of $F$ on $Y$ has a dense orbit; the general definition for reducing subgroups is given in Section \ref{S.Reduce}.  Our first main result is a characterization of trees that are not reduced by any non-trivial, proper free factor; we call these trees \emph{arational}.  

\begin{thm}
 Let $T \in \partial \cvn$.  The following are equivalent:
\begin{enumerate}
 \item [(i)] $T$ is arational,
 \item [(ii)] $T$ is indecomposable, and if $T$ is not free, then $T$ is dual to a measured geodesic lamination on a once-punctured surface with minimal and filling support.
\end{enumerate}

\end{thm}

The notion of indecomposability for trees was introduced in \cite{Gui08}; indecomposable trees in $\partial \cvn$ were studied in \cite{R10a} and \cite{CHR11}.  Associated to $T \in \partial \cvn$ is a (algebraic) lamination $L(T)$--a non-empty, closed, $\FN$-invariant, $\mathbb{Z}_2$-invariant subset of $\partial \FN \times \partial \FN \ssm \Delta.$; $L(T)$ encodes information about elements of $\FN$ with short translation length in $T$.  The structure of $L(T)$ is related to the structure of $T$; in \cite{CHR11} it was shown that if $T$ is free and indecomposable, then $L(T)$ is minimal up to adding finitely many $\FN$-orbits of diagonal leaves.  It follows from \cite{R10a} that if $T$ is free and indecomposable, then every sublamination of $L(T)$ is filling in the sense that it is not carried by any finitely generated subgroup of infinite index.  One sees that for $T$ dual a measured lamination on a once-punctured surface with minimal and filling support, then $L(T)$ contains a sublamination that is minimal and filling; in this case $L(T)$ is a symbolic coding for the corresponding singular foliation, and the referenced minimal and filling sublamination corresponds to the (coding of the) geodesic lamination.  

A \emph{current} is an $\FN$-invariant, $\mathbb{Z}_2$-invariant, Radon measure on $\partial \FN \times \partial \FN \ssm \Delta$; the support of a current is a lamination.  Kapovich and Lustig \cite{KL09a} constructed a continuous function $\langle \cdot , \cdot \rangle:\partial \cvn \times Curr_N \to \mathbb{R}$ and showed that $\langle T, \mu \rangle=0$ if and only if $Supp(\mu) \subseteq L(T)$\cite{KL10d}; here $Curr_N$ denotes the space of geodesic currents, equipped with the weak$^*$ topology.  The action of $Out(\FN)$ on the corresponding projective space $\mathbb{P}Curr_N$ is not minimal, but there is a unique minset $\mathbb{P}M_N \subseteq \mathbb{P}Curr_N$, where $Out(\FN)$ acts with dense orbits \cite{Kap06}.  Let $M_N$ denote the pre-image of $\mathbb{P}M_N$ in $Curr_N$.  We get the following ``unique duality'' result, which is useful for constructions; in particular, it is used in \cite{BR12}, where details of the proof are found, as part of an approach to give a description of the Gromov boundary of the complex of free factors of $\FN$:

\begin{cor}\cite{BR12}
 If a tree $T \in \partial \cvn$ is arational then the following holds: for any $T' \in \partial \cvn$ and any $\mu \in M_N$, if $\langle T, \mu \rangle =0=\langle T',\mu \rangle$, then $L(T)=L(T')$.  In particular, $T'$ is also arational.
\end{cor}

To complement our characterization of arational trees, we establish some results giving control over the set of all non-trivial proper free factors that reduce a given tree.  Say that a factor $F \leq \FN$ \emph{dynamically reduces} $T$ if there is an $F$-invariant subtree $Y \subseteq T$ that contains more than one point such that the action of $F$ on $Y$ has dense orbits.  If $F$ reduces $T$ but does not dynamically reduce $T$, then we say that $F$ \emph{peripherally reduces} $T$; in this case, $F$ fixes a point of $T$.  Let $\mathcal{R}(T)$ denote the set of all factors reducing $T$, and let $\mathcal{D}(T)$ denote the set of factors dynamically reducing $T$.  For any $T \in \partial \cvn$, there is a simplicial tree $T' \in \partial \cvn$ such that any subgroup fixing a point in $T$ fixes a point in $T'$, so peripheral factors are understood in any case.  

\begin{thm}
 Let $T \in \partial \cvn$.  If $T$ is factor reducible, then there is a finite subset $\mathcal{C}(T)=\{F^1,\ldots,F^r\} \subseteq \mathcal{R}(T)$ such that:
    \begin{enumerate}
     \item [(a)] for any $F' \in \mathcal{R}(T)$, there is $g \in \FN$ such that $F^j \leq (F')^g$ for some $j$,
     \item [(b)] there is a simplicial tree $T' \in \partial \cvn$ such that any subgroup fixing a point in $T$ also fixes a point in $T'$ and such that some element of each $F^j$ fixes a point in $T'$, 
     \item [(c)] the action of $F^j$ on its minimal invariant subtree in $T$ is mixing (maybe trivial); further, for any $i \neq j$ and any $g \in \FN$, $F^i \cap (F^j)^g$ fixes a point in $T$.
    \end{enumerate}
\end{thm}

Conditions (a) and (b) give control over the set $\mathcal{R}(T)$.   In \cite{BR12} these conditions are used to show that there is a number $L$ such that for any tree $T \in \partial \cvn$, the diameter of $\mathcal{R}(T)$ has diameter at most $L$ in the complex of free factors for $\FN$; thanks to the theorem, we just need to understand the free factors containing an elliptic element of $T'$, which is achieved using an argument with Whitehead's algorithm.  Conditions (a) and (c) combine to give that the collection $F^1, \ldots, F^r$ is canonical in some sense:

\begin{cor}
 Let $T$ and $\mathcal{C}(T)$ as in the Theorem, and let $\Phi \in Out(\FN)$.  If $T\Phi$ is equivariantly bi-Lipschitz equivalent to $T$, then $\Phi$ preserves the set of conjugacy classes $\{[F^i]|F^i \in \mathcal{C}(T)\}$.  
\end{cor}

Note that if $T\Phi$ is equivariantly bi-Lipschitz equivalent to $T$, then $\Phi$ certainly preserves $L(T)$; hence, we recover the following variant of a celebrated result of Bestvina-Handel \cite{BH92}, see also \cite{BFH97}:

\begin{thm}
 Let $\Phi \in Out(\FN)$.  One of the following holds:
  \begin{enumerate}
   \item [(i)] $\Phi$ is finite order,
   \item [(ii)] there is $k$, such that $\Phi^k$ fixes a conjugacy class of non-trivial proper free factors of $\FN$, or
   \item [(iii)] $\Phi$ preserves a minimal and filling lamination.
  \end{enumerate}

\end{thm}

\section{Basic Notions}

A metric space $(T,d)$ is called a \emph{tree} if for any $x,y \in T$, there is a unique topological arc $[x,y]$ connecting $x$ to $y$ and $[x,y]$ is isometric to $[0,d(x,y)] \subseteq \mathbb{R}$; every tree gets the metric topology.  If $x \neq y$, we call $[x,y]$ an \emph{arc}; use the term \emph{degenerate arc} to mean a point.  A tree is called \emph{finite} if it is a finite union of arcs.  The convex hull of three points in $T$ is called a \emph{tripod} if it is not an arc.

We consider trees that carry a (left) isometric action of a group $G$, which will always be a free group; this means that we have homomorphism $\rho:G \to Isom(T)$.  As usual, the representation is suppressed, and $g \in G$ is identified with the isometry $\rho(g)$.  A tree $T$ equipped with an isometric action of $G$ is called an $G$-\emph{tree}, and we sometimes denote this situation by $G \curvearrowright T$.  Two $G$-trees $T$ and $T'$ are identified if there is an equivariant isometry $T \to T'$.

For an element $g \in Isom(T)$, we set $l(g)=l_T(g):=\inf_{x \in T}d(gx,x)$ to be the \emph{translation length} of $g$; there are two sorts of isometries of $T$--$g$ is \emph{hyperbolic} if $l(g)>0$, and $g$ is \emph{elliptic} otherwise.  If $g$ is hyperbolic, then there is an isometrically embedded copy of $\mathbb{R}$ in $T$, denoted $A(g)$, and called the \emph{axis} of $g$, on which $g$ acts as a translation by distance $l(g)$.  If $g$ is elliptic, then $g$ fixes a point in $T$, and we let $A(g)$ stand for the fixed point set of $g$.  

Let $T$ be a $G$-tree, and let $H \leq G$ be finitely generated.  If $H$ contains a hyperbolic element, then there is a unique minimal $H$-invariant subtree $T_H \subseteq T$; $T_H$ is the union of all axes of hyperbolic elements of $H$.  In general, $T$ is called \emph{minimal} if there is no proper subtree $T' \subsetneq T$ that is $G$-invariant; if $G$ is finitely generated, minimality amounts to $T_G=T$.  For $A \subseteq T$, we use $Stab(A)$ to denote the (setwise) \emph{stabilizer} of $A$, \emph{i.e.} $Stab(A)=\{g \in G|gA=a\}$.

A subtree $K \subseteq T$ is called a \emph{supporting subtree} if for any arc $I \subseteq T$, there are $g_1,\ldots, g_k \in G$ such that $I \subseteq g_1K \cup \ldots \cup g_kK$.  If $G$ is finitely generated and if $T$ is minimal, then $T$ always contains a finite supporting subtree; for example, let $x \in T$ be any point, and let $K$ be the convex hull in $T$ of the set containing $x$ and its images under a symmetric generating set for $G$.  

Let $\FN$ denote the rank $N$ free group.  A minimal $\FN$-tree $T$ is \emph{very small} if for any arc $I \subseteq T$, $Stab(I)$ either is trivial or is a maximal cyclic subgroup of $\FN$ and if the stabilizer of any tripod is trivial.  In this paper, we only consider very small trees that are not free and simplicial; let $\partial \cvn$ denote this subspace of such very small $\FN$-trees; identifying elements of $\partial \cvn$ that are equivariantly homothetic gives the boundary of the Culler-Vogtmann Outer space \cite{CL95, BF94}.  Most of the time, we will be concerned with very small trees with \emph{dense orbits}; this means that some (equivalently every) orbit of $T$ is dense in $T$.  If $T$ is very small and has dense orbits, then arc stabilizers are trivial \cite{LL03}.  In this case, for any finitely generated $H \leq \FN$, there is a unique minimal $H$-invariant subtree, which we denote by $T_H$; if $H$ contains a hyperbolic element, $T_H$ is defined above, else $T_H$ is the unique fixed point of $H$.  Note that the action of a group on a point has dense orbits.  

Let $\partial^2 \FN$ be defined as $\partial^2 \FN:= \partial \FN \times \partial \FN \ssm \Delta$, where $\partial \FN$ is the Gromov boundary, with the usual topology, of $\FN$ and where $\Delta$ denotes the diagonal; $\partial^2 \FN$ gets the induced topology.  The action of $\FN$ on $\partial \FN$ gives an action on $\partial^2 \FN$; we also have the involution $\iota$ exchanging the factors. A \emph{lamination} is a non-empty, closed, $\FN$-invariant, $\iota$-invariant subset $L \subseteq \partial^2 \FN$; the elements of a lamination are called \emph{leaves}.  Laminations generalize symbolic codings of foliations on surfaces.  Associated to $T \in \partial \cvn$ is a lamination $L(T)$, which is defined as follows: put $L_{\epsilon}(T):=\{(g^{-\infty},g^{\infty})|l(g) < \epsilon\}$, so $L_{\epsilon}(T)$ is a lamination; and set $L(T):=\cap_{\epsilon}L_{\epsilon}(T)$.  

If $H \leq \FN$ is finitely generated, then M. Hall's theorem gives that $H$ is a virtual retract, hence $H$ is quasi-convex in $\FN$; $\partial^2 H$ embeds in $\partial^2 \FN$.  Say that $H$ \emph{carries} a leaf of $L(T)$ if there is $l \in L(T)$ such that $l \in \partial^2 H$.  The folowing is an easy exercise:

\begin{lemma}\label{L.Carry}
 Let $T \in \partial \cvn$, and let $H \leq \FN$ be finitely generated.  Then $H$ carries a leaf of $L(T)$ if and only if one of the following occurs:
\begin{enumerate}
 \item [(i)] some non-trivial element of $H$ fixes a point in $T$, or
 \item [(ii)] the action $H \curvearrowright T_H$ is not discrete.
\end{enumerate}
\end{lemma}

Condition (ii) is, of course, more interesting, and is the focus of this paper.

\section{Transverse Families}

Let $T$ be an $\FN$-tree.  An $\FN$-invariant collection $\mathscr{Y}=\{Y_v\}_{v \in V}$ of non-degenerate subtrees $Y_v \subseteq T$ is called a \emph{transverse family} if whenever $Y_v \neq Y_{v'}$, one has that $Y_v \cap Y_{v'}$ contains at most one point.  Note that given a transverse family $\mathscr{Y}$, one has that the collection of closures $\overline{\mathscr{Y}}=\{\overline{Y}_v\}_{v \in V}$ is also a transverse family.  If $\mathscr{Y}$ is a transverse family in $T$ and if, additionally, $\mathscr{Y}=\overline{\mathscr{Y}}$ and each finite arc $I \subseteq T$ can be covered by finitely many elements of $\mathscr{Y}$, then $\mathscr{Y}$ is a \emph{transverse covering} of $T$, as defined in \cite{Gui08}.  We call the element of a transverse covering \emph{vertex trees} or \emph{vertex actions}, and when $T$ has a transverse covering, we say that $T$ \emph{splits} or that $T$ is a \emph{graph of actions}.

A transverse family in a tree $T$ should be thought of as a reduction of $T$.  Here is a motivating example (also see the Appendix): let $S$ be a surface, where we have an identification $\pi_1(S)=\FN$, and suppose that $S$ is equipped with a measured lamination $(L,\mu)$, where $\mu$ is assumed to have full support.  Set $T=T_{(L,\mu)}$ to be the $\FN$-tree dual to the measured lamination $(L,\mu)$; see \cite{MS84, BF95}.  If $L$ is not minimal, then the set of subtrees of $T$ corresponding to a particular minimal component is a transverse family in $T$; further, the collection of all subtrees of $T$ that come from a minimal component of $L$ is a transverse covering of $T$.  This is the intuition: transverse families are a generalization of sublaminations, and the subsurfaces supporting them, of measured laminations.  

Let $\mathscr{Y}$ be a transverse covering of $T$; associated to $\mathscr{Y}$ is a (minimal) simplicial $\FN$-tree $S$, called the \emph{skeleton} of $\mathscr{Y}$, which is defined as follows \cite{Gui08}.  The vertex set $V(S)$ of $S$ is the union of $\mathscr{Y}$ with the set of intersection points $x \in Y \cap Y'$, called \emph{attaching points}, for $Y \neq Y' \in \mathscr{Y}$; as $\mathscr{Y}$ is a transverse family, if $Y \cap Y' \neq \emptyset$, then there is exactly one point in the intersection if $Y \neq Y'$.   Connect $x \in V(S)$ to $Y \in V(S)$ by an edge if and only if $x \in Y$.  The space $S$ carries an obvious $\FN$-action, and the definition of $S$ allows for an easy check that $S$ is indeed a tree; minimality of $S$ follows from minimality of $T$; see \cite{Gui08} or the Appendix for details.  

Note that edge stabilizers in $S$ are of the form $Stab_Y(x)=Stab(x) \cap Stab(Y)$, where $x \in Y$ is an attaching point.  In particular, $S$ contains an edge with cyclic (resp. trivial) stabilizer if and only if there is $Y \in \mathscr{Y}$ and an attaching point $x \in Y$ such that $Stab_Y(x)$ is cyclic (resp. trivial); this is always satisfied when $T$ has cyclic point stabilizers (resp. $T$ is free), but can happen in more general situations (see Appendix).  

The simplicial trees $S$ that arise need not be very small (see Appendix).  If $T$ has a transverse covering whose skeleton is very small, then we say that $T$ has a \emph{very small splitting}; and if $T$ has a transverse covering whose skeleton contains an edge with trivial stabilizer, then we say that $T$ has a \emph{free splitting}.  We immediately have the following:

\begin{lemma}
 Let $T \in \partial \cvn$ have transverse covering $\mathscr{Y}$ with skeleton $S$.  If $S$ is a free splitting of $\FN$, then for any $Y \in \mathscr{Y}$, $Stab(Y)$ is contained in a proper free factor of $\FN$.  
\end{lemma}

As noted above, the hypotheses are satisfied by any free $T$ admitting a transverse covering.

\subsection{Residuals and Intersections of Transverse Families}

%Let $\mathscr{Y}$ and $\mathscr{Y}'$ be transverse families in $T$.  Say that $\mathscr{Y}'$ \emph{refines} $\mathscr{Y}$ if every element of $\mathscr{Y}'$ is contained in an element of $\mathscr{Y}$ and every element of $\mathscr{Y}$ contains some element of $\mathscr{Y}'$; this gives a partial order on transverse families in $T$ by declaring $\mathscr{Y}' \leq \mathscr{Y}$.  

If $X \subseteq T$ is a subtree of $T$ and if $\mathscr{X}$ is a transverse family in $T$, we let $X \smf \mathscr{X}:=X \ssm \cup I$, where $I$ runs over all arcs $I \subseteq X_i \in \mathscr{X}$ with $I \cap X$ non-degenerate.  Call $X \smf \mathscr{X}$ the \emph{residual} of $\mathscr{X}$ in $X$.  Say that $\mathscr{X}$ has \emph{non-degenerate} residual in $X$ if some component of $X \smf \mathscr{X}$ is non-degenerate, \emph{i.e.} if there is a non-degenerate arc $J \subseteq X$ that does not intersect any element of $\mathscr{X}$ non-degenerately; otherwise, say that $\mathscr{X}$ has \emph{degenerate} residual in $X$.  

%, then define $\mathscr{X} \smf \mathscr{Y}$ to be the family $\{X_i \smf \mathscr{Y}|X_i \in \mathscr{X}\}$  urther, define $\mathscr{X} \Cap \mathscr{Y}$ to be the family of non-degenerate subtrees $Z \subseteq T$ such that there are $X_i \in \mathscr{X}$ and $Y_j \in \mathscr{Y}$ with $Z=X_i \cap Y_j$.  Finally,

If $\mathscr{Y}$ is another transverse family in $T$, define $\mathscr{X} \wedge \mathscr{Y}$ to be the family of subtrees of $T$ consisting of non-degenerate intersections $X \cap Y$ for $X \in \mathscr{X}$ and $Y \in \mathscr{Y}$.  Here is a simple observation:

\begin{lemma}\label{L.CommonRefinement}
 Let $T \in \partial \cvn$, and let $\mathscr{X}$ and $\mathscr{Y}$ be transverse families in $T$.  
\begin{enumerate}
 \item [(i)] $\mathscr{X} \wedge \mathscr{Y} = \emptyset$ if and only if $\mathscr{X} \cup \mathscr{Y}$ is a transverse family, and
 \item [(ii)] if $\mathscr{X} \wedge \mathscr{Y} \neq \emptyset$, then $\mathscr{X} \wedge \mathscr{Y}$ is a transverse family.
\end{enumerate}  
  
\end{lemma}

\begin{proof}
 Statement (i) is obvious.  For statement (ii), note that invariance of $\mathscr{X}$ and $\mathscr{Y}$ ensure that the collection $\mathscr{X} \wedge \mathscr{Y}$ is invariant.  If $X \cap Y, gX \cap hY \in \mathscr{X} \wedge \mathscr{Y}$ are such that $X \cap Y \cap gX \cap hY$ is non-degenerate, then since $\mathscr{X}$ and $\mathscr{Y}$ are transverse families, we have that $gX=X$ and $hY=Y$; hence $X \cap Y = gX \cap hY$.  
\end{proof}

We note that it also follows that for any $X \in \mathscr{X}$, the collection of those memebers of $\mathscr{X} \wedge \mathscr{Y}$ that are contained in $X$ is a transverse family for the action $Stab(X) \curvearrowright X$.

\section{Structure Theory for Special classes of trees}

In this head, we consider two classes of trees for which there exist structural results; these results will be useful in obtaining information about more general trees in $\partial \cvn$.

\subsection{Levitt's Coarse Structure Theorem}

Consider a tree $T \in \partial \cvn$ that is not simplicial; this means that the orbit of some point in $T$ has an accumulation point.  In this case, \cite{Lev94} gives that $T$ has a transverse covering, whose members either are simplicial or have dense orbits with respect to the action of the stabilizer.  The simplicial actions need not be minimal with respect to the action of the stabilizer; it is possible to have a trivial action on a segment.  Choosing the actions with dense orbits to be maximal and choosing the simplicial actions to be edges, one can make this decomposition unique; here are the details:

%\begin{lemma}\cite{Lev94}\label{L.Levitt}
%Let $T$ be a very small $\FN$-tree.  If $T$ does not have dense orbits and if $T$ is not simplicial, then $T$ 
%\end{lemma}

\begin{lemma}\cite{Lev94}\label{L.GoodGraph}
 Let $T$ be a very small $\FN$-tree.  If $T$ does not have dense orbits and if $T$ is not simplicial, then $T$ has a very small splitting with vertex actions either:
\begin{enumerate}
 \item [(i)] $H \curvearrowright Y_H$, having dense orbits, where $Y_H$ is maximal, or
 \item [(ii)] $E \curvearrowright e$, where $e$ is a segment whose interior does not contain a branch point of $T$ and which is maximal with respect to this property.
\end{enumerate}
Further, this transverse covering is unique.  
\end{lemma}

%Since we are interested in a dynamical-algebraic decomposition of actions in $\partial \cvn$, we record some facts about the algebra around tree that are not simplicial and do not have dense orbits.  

%DOES THIS ACCOMPLISH ANYTHING?
%\begin{cor}
% Let $T \in \partial \cvn$ and assume that $T$ is neither simplicial nor with dense orbits.  Let $H_1 \curvearrowright Y_1, \ldots, H_r \curvearrowright Y_r$ be the collection of vertex actions with dense orbits as in Lemma \ref{L.GoodGraph}.  There is a very small simplicial tree in which each $H_i$ is a vertex stabilizer.
%\end{cor}

%\begin{proof}
% As each $Y_i$ is maximal, each $Y_i$ is closed.  For every $g \in \FN$, collapse $gY_i$ to a point.  The resulting space is the desired simplicial tree.
%\end{proof}

%In particular, we can apply Lemma \ref{L.Simp} to get information about the subgroups $H_i$.  Indeed, given any $T \in \partial \cvn$, there is a simplicial tree $T' \in \partial \cvn$ such that if $H \leq \FN$ fixes a point in $T$, then $H$ fixes a point in $T'$ as well \cite{BF94}, \emph{i.e.} understanding point stabilizers in $T$ is closely related to understanding vertex stabilizers in certain simplicial $\FN$-trees.  

We turn to another special case of actions in $\partial \cvn$.

\subsection{Geometric Trees}

We review geometric trees; the reader is assumed to have some familiarity with this subject; see \cite{BF95} for details.  For this section, we fix a basis $\mathscr{B}=\{b_1,\ldots,b_N\}$ for $\FN$.  Let $T \in \partial \cvn$, and let $K \subseteq T$ be a finite supporting subtree.  The restrictions of the elements of $\mathscr{B}$ to $K$ give a collection of partial isometries of $K$
$$
b_i:b_i^{-1}K \cap K \to K \cap b_iK
$$

The suspension of $(K,\mathscr{B})$, denoted $\mathscr{K}$, is got by gluing $(b_i^{-1}K\cap K)\times [0,1]$ to $K$: glue $b_i^{-1}K\cap K \times \{0\}$ to $K$ via inclusion, and glue $b_i^{-1}K \cap K \times \{1\}$ to $b_i(b_i^{-1}K \cap K)=K \cap b_iK$ so that in the adjunction space $\mathscr{K}$, one has for $x \in b_i^{-1}K \cap K$ that $\{1\} \times \{x\} =b_i(x)$.  The foliations given by $\{pt.\} \times [0,1]$ on the sets $(b_i^{-1}K \cap K) \times [0,1]$ extend to a singular foliation on $\mathscr{K}$; the leaves of this foliation correspond to orbits in the pseudo group generated by the restrictions of the elements of $\mathscr{B}$ to Borel subsets of $K$.  Note that we usually just consider maximal restrictions of elements of $\FN$ to $K$.

There is a transverse measure on $\mathscr{K}$ given by the Lebesgue measure on $K$, which is inherited by the Lebesgue measure on arcs of $T$.  Note that $\pi_1(\mathscr{K})=\FN$, and so $\FN$ acts on the universal cover $\tilde{\mathscr{K}}$ by deck transformations.  Collapsing each leaf of the lifted foliation on $\tilde{\mathscr{K}}$ to a point gives a tree $T_{K}$, equipped with an isometric action of $\FN$.  The tree $T_K$ comes with a more-or-less obvious (surjective) equivariant map $f_K:T_K \to T$--this comes from the assumption that $K$ is a supporting subtree for $T$  ($K$ is clearly a supporting subtree of $T_K$).  

The map $f_K:T_k \to T$ is called a \emph{resolution} of $T$; we say that this resolution corresponds to the choices $\mathscr{B}$ and $K$.  If in $T$ there is a finite subtree $K$ such that the map $f_K:T_K \to T$ is an isometry, then the resolution $f_K:T_K \to T$ is called \emph{exact}, and $T$ is called \emph{geometric}.  If $f_K:T_K \to T$ is exact, then for another basis  $\mathscr{B}'$, then there is a finite subtree $K'$ such that $f_{K'}:T_{K'} \to T$ is exact as well.  Geometric trees are special: the dynamical structure of a geometric tree $T$ can be understood by studying the foliation on $\mathscr{K}$, which completely encodes the action $\FN \curvearrowright T$.  

\subsection{Structure of Geometric Trees}

The key result concerning the dynamical structure of geometric actions is a theorem of Imanishi \cite{Ima79}.  Imanishi's theorem was rediscovered by Morgan-Shalen \cite{MS84}, and was proved in the present context by Levitt \cite{GLP94}; it states that given a finite 2-complex $A$, equipped with a codimension-1 singular measured foliation, one is able to cut $A$ along certain subsets of singular leaves to get a new complex where every leaf either is compact or is locally dense \cite{GLP94, BF95}; the key point is that no leaf closure is a Cantor set.  A simplicial tree with finitely presented vertex and edge groups is geometric, so all simplicial trees in $\partial \cvn$ are geometric.  We state the following consequnce of Imanishi's theorem in the present context; for the statement we fix some basis for $\FN$ and use the notation from above.  

\begin{prop}\label{P.WeakImanishi}\cite[Proposition 1.25]{Gui08}
 Let $T \in \partial \cvn$ be geometric, and assume that $T$ is not simplicial.  Either there is a finite subtree $K \subseteq T$, such that the foliation on $\mathscr{K}$ is minimal and such that $f_K:T_K \to T$ is exact, or $T$ has a transverse covering $\mathscr{Y}$, such that every $Y \in \mathscr{Y}$ either is a simplicial edge or is dual to a 2-complex equipped with a minimal measured foliation.  
\end{prop}

If $\mathscr{K}$ is minimal but not \emph{pure}, then there is transverse covering of $T$; see \cite{BF95}, \cite{Gui08}, or the discussion below.  Proposition \ref{P.WeakImanishi} is stronger than Lemma \ref{L.GoodGraph} applied to geometric trees, as we get more infomation about the vertex actions with dense orbits.  The example to keep in mind is a surface carrying a measured lamination, whose underlying lamination is not minimal; see Appendix.  

%One can say even more regarding the algebraic structure of geometric trees; this is the content of Rips' Theory \cite{BF95, GLP94}.  One consequence is that if $S$ is the skeleton of the graph of actions structure on a geometric tree $T$ given by Proposition \ref{P.WeakImanishi}, then all vertex and edge stabilizers in $S$ are finitely generated.  Geometric trees are not typical in $\partial \cvn$, so up to this point, we only have Lemma \ref{L.GoodGraph}, which applies in full generality and reduces the problem of understanding trees in $\partial \cvn$ to that of understanding trees with dense orbits.  On the other hand, geometric trees are useful for studying arbitrary trees in $\partial \cvn$; the point, which will be used later, is that arbitrary trees in $\partial \cvn$ can be approximated in a very strong way by geometric trees.

As we will soon see, the above proposition also gives information about the possible residuals of transverse families in geometric tree with dense orbits.  Indeed, let $T \in \partial \cvn$ have dense orbits, and suppose that $T$ is geometric.  Proposition \ref{P.WeakImanishi} gives that $T$ has a transverse covering by subtrees that are dual to 2-complexes carrying minimal measured foliations.  One might expect such subtrees to be ``irreducible'' in some sense, and this is the case.  In order to formalize this, we will need a few more definitions and results, which will be collected in the next section.

\section{Mixing Properties of Trees}

Dense orbits is a very weak notion of ``irreducibility'' for a tree.  As before, the examples to keep in mind are surface trees; more details on the following appear in the Appendix.  Consider a surface $S$, equipped with a measured lamination $(L,\mu)$.  It is an easy exercise to check that as long as $L$ has no simple closed curve components, the dual tree $T_{(L,\mu)}$ has dense orbits; in particular, even if $_{(L,\mu)}$ has dense orbits, $L$ can have many minimal components, and we ought to be able to detect such a situation.  Hence, we introduce some notions refining dense orbits.

Let $T \in \partial \cvn$.  The action $\FN \curvearrowright T$ is called \emph{mixing} if for any non-degenerate arcs $I, J \subseteq T$, there are $g_1, \ldots, g_r \in \FN$ such that $J \subseteq g_1I \cup \ldots \cup g_rI$.  The action $\FN \curvearrowright T$ is \emph{indecomposable} if $T$ is mixing and if the elements $g_i$ can be chosen so that $g_iI \cap g_{i+1}I$ are non-degenerate; note that in the definition we do not require that $g_iI \cap J$ be non-degenerate for each $i$.  The notion of mixing for group actions on trees was introduced by Morgan \cite{Mor88}; Guirardel introduced the stronger notion of indecomposability in \cite{Gui08}, where he observed the following (see Proposition 1.25 and Remark 1.29 of \cite{Gui08}:

\begin{lemma}\cite{Gui08}
 Let $T \in \partial \cvn$.  If $T$ admits an exact resolution by a pure band complex carrying a minimal measured foliation, then $T$ is indecomposable. 
\end{lemma}

Whence, we get the following:

\begin{cor}\cite[Proposition 1.25]{Gui08}\label{C.GeomGraph}
 Let $T \in \partial \cvn$.  If $T$ is geometric, then either $T$ is indecomposable, or else $T$ has a transverse covering $\mathscr{Y}$, where each $Y \in \mathscr{Y}$ either is indecomposable or is a simplicial edge.  In particular, if $T$ has dense orbits, then $T$ has a transverse covering by indecomposable trees.
\end{cor}

To understand the dynamical structure of a tree $T \in \partial \cvn$, we will find transverse families whose members have some mixing properties.  It will be helpful to have a characterization of the above mixing properties, or rather their negations.  We will use the following discussion about the implications of the negations of mixing and indecomposable; we also establish some notation.

\begin{discussion}\label{D.BuildTF}
 We will use the notation presented here in the sequel.  Let $T \in \partial \cvn$, and let $I \subseteq T$ be a non-degenerate subtree.  Define $X_I$ to be the union of all arcs $J$ of $T$ such that there are $g_1,\ldots,g_r \in \FN$ such that $J \subseteq g_1I \cup \ldots \cup g_rI$ and such that $g_iI \cap g_{i+1}I \neq \emptyset$.  Define $Y_I$ similarly but with the additional requirement that $g_iI \cap g_{i+1}I$ be non-degenerate.  

 Note that $T$ is not mixing if and only if there is a non-degenerate arc $I \subseteq T$ such that $X_I \neq T$; $T$ is not indecomposable if and only if there is a non-degenerate arc $I \subseteq T$ such that $Y_I \neq T$.  By construction, we have that for $g \in \FN$, if $gX_I \neq X_I$, then $gX_I \cap X_I = \emptyset$; and if $gY_I \neq Y_I$, then $gY_I \cap Y_I$ is degenerate.  In particular, the collection $\mathscr{X}_I=\{gX_I\}_{g \in \FN}$, respectively $\mathscr{Y}_I=\{gY_I\}_{g \in \FN}$, is a transverse familiy whenever $X_I$, respectively $Y_I$, is a proper subtree of $T$.
\end{discussion}

We immediately get the following, which is observed in \cite{R10a}. 

\begin{lemma}\label{L.IndecompChar}
 A tree $T \in \partial \cvn$ is indecomposable if and only if there is no transverse family in $T$.
\end{lemma}

Here is another simple consequence.

\begin{lemma}\label{L.MixGraph}
 Let $T \in \partial \cvn$, and suppose that $T$ is mixing.  If $T$ is not indcomposable, then $T$ has a transverse covering; further, every transverse covering of $T$ contains exactly one orbit.
\end{lemma}

\begin{proof}
 Since $T$ is not indecomposable, Discussion \ref{D.BuildTF} gives a transverse family $\mathscr{Y}$ for $T$.  Let $I \subseteq Y \in \mathscr{Y}$ be an arc.  Since $T$ is mixing, for any arc $J \subseteq T$, there are $g_1,\ldots,g_r \in \FN$ such that $J \subseteq g_1I \cup \ldots \cup g_rI$; both conclusions are immediate.
\end{proof}

Indecomposable trees in $\partial \cvn$ were investigated in \cite{R10a}; we summarize the main results of that paper in the following:

\begin{prop}\label{P.Indecomp}
 Let $T \in \partial \cvn$ be indecomposable, and let $H \leq \FN$ be finitely generated.  The action $H \curvearrowright T_H$ is non-simplicial if and only if $H$ has finite index in $\FN$.  If in addition $T$ is free, then $H$ carries a leaf of $L(T)$ if and only if $H$ has finite index in $\FN$. 
\end{prop}

The proposition essentially says that the only dynamically interesting subactions of an indecomposable actions are virtually the action itself.  We get the following consequence:

\begin{lemma}\label{L.GeomResidual}
 Let $T \in \partial \cvn$ have dense orbits, and suppose that $T$ is geometric.  If $\mathscr{Y}$ is a transverse family in $T$ that is not a transverse covering of $T$, then the residual of $\mathscr{Y}$ in $T$ is non-degenerate.
\end{lemma}

In particular, if $\mathscr{Y}$ is a transverse family in a geometric tree $T$ such that $T \smf \mathscr{Y}$ contains no non-degenerate arc, then $\overline{\mathscr{Y}}$ is a transverse covering for $T$.

\begin{proof}
 Let $\mathscr{X}$ be the transverse covering of $T$ by indecomposable subtrees given by Corollary \ref{C.GeomGraph}.  Lemma \ref{L.CommonRefinement} gives that $\mathscr{X} \wedge \mathscr{Y}$ is a transverse family, and for any $X \in \mathscr{X}$, the collection of members of $\mathscr{X} \wedge \mathscr{Y}$ contained in $X$ is a transverse family for $Stab(X) \curvearrowright X$.  On the other hand, Lemma \ref{L.IndecompChar} gives that the elements of $\mathscr{X}$ contain no transverse families.  It follows that every member of $\mathscr{Y}$ is a union of members of $\mathscr{X}$.  To conclude, we just need to see that every proper transverse family contained in $\mathscr{X}$ has non-degenerate residual, but this is obvious for any transverse covering.
\end{proof}

\section{Reducing Systems}\label{S.Reduce}

%Our goal is a dynamical-algebraic structure theory for very small $\FN$-trees.  Our general approach is typical throughout mathematics: describe some class of ``prime'' objects, describe how to combine these objects to build an arbitrary object, and then show that this constuction is canonical.  We have already seen the class of prime objects; they are indecomposable trees.  To show that indcomposable trees are indeed the basic object, we need to get an understanding of the algebraic implications of indecomposability.  Some important dynamical-algebraic properties of indecomposable trees in $\partial \cvn$ appear in Proposition \ref{P.Indecomp}.  To understand the stucture of a general very small $\FN$-tree, we need understand ...... BLAH BLAH  

For a subtree $Y \subseteq T \in \partial \cvn$, say that $Y$ \emph{generates a transverse family} if the collection of translates $\{gY\}_{g \in \FN}$ is a transverse family.  

\begin{defn}
A \emph{reducing system} for $T$ is a pair $\mathcal{R}=(\mathcal{D},\mathcal{P})$ satisfying the following:
\begin{enumerate}
 \item [(i)] $\mathcal{D}$ is a collection of conjugacy classes non-trivial, proper, finitely generated subgroups of $\FN$.  For any $H \in \mathcal{D}$, $T_H$ is not a point, the action of $H$ on $T_H$ has dense orbits, and $T_H$ generates a transverse family,
 \item [(ii)] $\P$ is a collection of conjugacy classes of non-trivial, proper, finitely generated subgroups of $\FN$.  For any $K \in \P$, $K$ fixes a point of $T$.
\end{enumerate}
\end{defn}

The abuse of language in the definition is without ill consequence; even though $T_H$ is meaningless for a conjugacy class of subgroups $H$, any two representatives of this conjugacy class give  minimal trees that equivariantly isometric.  We will continue with this imprecise language, since it we feel there is little chance for confusion.  Additionaly for conjugacy classes of subgroups $[H]$ and $[K]$ we will say that $[H]$ \emph{is contained in} $[K]$ if there is $H' \in [H]$ with $H \leq K$; and we will say that $[H]$ (non-trivially) \emph{intersects} $[K]$ if there is $H' \in [H]$ such that $H' \cap K \neq \{1\}$.  We will also take liberties in dropping the ``$[\cdot]$'' when referring to conjugacy classes of subgroups.

A reducing system $\mathcal{R}$ is \emph{trivial} if both $\D$ and $\P$ are empty; otherwise $\R$ is \emph{non-trivial}.  We call the elements of $\D$ \emph{dynamical subgroups}, and we call the elements of $\P$ \emph{peripheral subgroups}; more generally, we call the elements of $\D \cup \P$ \emph{reducing subgroups}.  Note that if $H \leq \FN$ has finite index, then $T_H=T$; since elements of a transverse family must be proper subtrees of $T$, it must be the case that every element of a reducing system has infinite index in $\FN$.  It is possible to obtain a canonical collection of peripheral subgroups from the following result of Jiang, which is classical for the subject; more information was given by \cite{GJLL98}.  

\begin{lemma}\label{L.CanonicalPeripheral}\cite{Ji91}
 Let $T \in \partial \cvn$ have dense orbits.  There are finitely many orbits of points of $T$ with non-trivial stabilizer, and the rank of every point stabilizer $T$ is at most $N$.  
\end{lemma}

The collection of peripheral subgroups given by Lemma \ref{L.CanonicalPeripheral} is canonical in the sense that any peripheral reducing subgroup is contained in one of them.  The issue of understanding potential elements of $\D$ and their minimal trees is more difficult and is our main focus. 

It should be noted that simply finding a subgroup $H$ such that $T_H$ has dense orbits and generates a transverse family is not completely satisfactory; indeed, replacing $H$ with any of its finite index subgroups will give the same situation, so we should look for such subgroups $H$ that are at least maximal with respect to some mixing condition.  

\begin{lemma}\label{L.Commensurator}
 Let $H \leq \FN$ be fintely generated, and suppose that $H$ does not fix a point of $T$ and that $T_H$ has dense orbits.  Then $H$ has finite index in $Stab(T_H)$; in particular, $Stab(T_H)$ is finitely generated.
\end{lemma}

Although we will not need it, we note that Lemma \ref{L.Commensurator} holds without restricting $T_H$ to have dense orbits; one uses the Stalling folding machinery to handle the case that $T_H$ is simplicial, and the general case follows easily from this and Lemma \ref{L.Commensurator} using Lemma \ref{L.GoodGraph}.  

\begin{proof}
As $H$ does not fix a point of $T$, $T_H$ is infinite and is the union of axes of hyperbolic elements of $H$.  Let $K \leq Stab(T_H)$ be any finitely generated subgroup of $Stab(T_H)$ that contains $H$.  It was noted in \cite{R10a} that if $L \leq \FN$ is finitely generated with has infinite index, then for any very small tree $Y$, if $Y_L$ has dense orbits, then $Y_L$ is a proper subtree of $Y$.  Hence, if $H$ had infinite index in $K$, then $T_H$ would be a proper subtree of $T_K$, but $K \subseteq Stab(T_H)$.  So, $H$ has finite index in $K$.

\end{proof}

In light of Lemma \ref{L.Commensurator}, it makes sense to adopt the convention that if $H \leq \FN$ is finitely generated and dynamically reduces $T$, then we assume that $H=Stab(T_H)$.

We have:

\begin{cor}\label{C.FactorStab}
 If $F \leq \FN$ is a free factor, then $F=Stab(T_F)$.  
\end{cor}

\begin{proof}
 $F$ has no finite extensions in $\FN$; apply Lemma \ref{L.Commensurator}.
\end{proof}

We collect more technical facts that allow us to find dynamically reducing subgroups.  If $T \in \partial \cvn$ has dense orbits and if $H$ carries a leaf of $L(T)$, then either some non-trivial element of $H$ fixes a point of $T$ or else the action $H \curvearrowright T_H$ is not discrete.  In the latter case, Lemma \ref{L.GoodGraph} gives that $T_H$ has a transverse covering by trees with dense orbits and simplicial edges.  Since arc stabilizers in $T$ are trivial, this splitting of $T_H$ is free, and so there is a free factor $H'$ of $H$ that acts with dense orbits on its minimal subtree.  If additionally $T_H$ generates a transverse family in $T$, then $T_{H'}$ will generate a transverse family in $T$ as well, since $T_{H'}$ generates a transverse family in $T_H$; in other words, $H'$ dynamically reduces $T$.  We now focus on the collection of free factors of $\FN$; more general reducing systems will be treated in later work if they are needed for applications.

\subsection{Factor Reducing Systems}

With a view toward the complex of free factors, we introduce a special class of reducing systems.  A \emph{factor reducing system} for $T \in \partial \cvn$ is a reducing system $\R=(\D,\P)$, where each element of $\D \cup \P$ is a free factor of $\FN$.  We begin by recalling some standard facts about free factors of $\FN$:

\begin{lemma}\label{L.FactorProperties}
 Let $F,F' \leq \FN$ be free factors.
\begin{enumerate}
 \item [(i)] $F \cap F'$ is a free factor of both $F$ and $F'$,
 \item [(ii)] any free factor of $F$ is a free factor of $\FN$,
 \item [(iii)] free factors satisfy the ascending and descending chain conditions.
\end{enumerate}
\end{lemma}

Our study of factor reducing systems is simplified by the following result, which is the main technical lemma from \cite{R10a}.

\begin{lemma}\label{L.FactorFamilies}
 Let $T \in \partial \cvn$ have dense orbits, and let $F' \leq \FN$ be a non-trivial, proper free factor.  If $T_{F'}$ has dense orbits, then $T_{F'}$ generates a transverse family.
\end{lemma}

%Lemma \ref{L.FactorFamilies} will be an engine for understanding factor reducing systems.  After collecting a few more tools, we will be able to obtain canonical factor reducing systems for trees in $\partial \cvn$.  One immediate consequence of the lemma is that if in $T$ there is a proper factor acting with dense orbits on its minimal subtree, then there is a non-trivial factor reducing system for $T$.  

Lemma \ref{L.FactorFamilies} along with the discussion following Corollary \ref{C.FactorStab} gives:

\begin{cor}\label{C.CarryToReduce}
 Let $T \in \partial \cvn $ have dense orbits; let $F \leq \FN$ be a proper factor, and assume that $F$ carries a leaf of $L(T)$.  If $T_F$ is simplicial, then there is a factor $F' \leq F$ that peripherally reduces $T$.  If $T_F$ is not simplicial, then there is a proper factor $F' \leq F$ such that $F'$ dynamically reduces $T$.
\end{cor}

Our goal will be to understand minimal factors acting with dense orbits on their minimal subtrees and to characterize trees for which all factor reducing systems are trivial.  Before doing this, we will need more technical lemmas; our approach is aided by considering a simplified ergodic theory for trees.

\section{Measures on Trees}

The class of objects discussed here might be thought of as one dimensional measures on trees.  For $T \in \partial \cvn$, a \emph{measure} on $T$ is a collection $\mu=\{\mu_I\}_{I \subseteq T}$ of finite, positive Borel measures on the finite arcs $I \subseteq T$.  It is required that for $I \subseteq J \subseteq T$, one has compatibility with resprect to restriction: $\mu_J|_I=\mu_I$.  All of our trees come with an action of $\FN$, and we will want measures to be \emph{invariant}; this means that for $g \in \FN$ and for a Borel set $X \subseteq I$, we have that $\mu_I(X)=\mu_{gI}(gX)$.  Let $\mathcal{M}(T)$ denote the set of invariant measures on $T$; $\mathcal M(T)$ has an obvious $\mathbb{R}_{\geq 0}$-linear structure and is convex.  

Measures on trees were defined by Paulin in his habilitation thesis and were further studied by Guirardel in \cite{Gui98}, where they were used to obtain informaiton about the dynamics of $Out(\FN)$ acting on $\partial \cvn$.  The definition of measure given here is what one gets by transporting the notion of a transverse measure on a surface lamination to the dual tree.  We collect some results from \cite{Gui98}, to which the reader is referred for more details.

As $T$ is an $\mathbb{R}$-tree, every finite arc of $T$ is isometrically identified with a finite segment in $\mathbb{R}$ and, hence, comes with Lebesgue measure.  Since $\FN$ acts isometrically on $T$, the measure on $T$ corresponding to this collection of Lebesgue measures if invariant.  We denote this measure by $\mu_{amb}(T)$, or just $\mu_{amb}$ when $T$ is understood.  One should think of $\mu_{amb}$ as the ambient measure on $T$.  

Measures on trees are locally defined objects; to define the usual measure-theoretic notions here, one works locally on a finite supporting subtree.  Let $\mu$ be an invariant measure on $T$, and let $K \subseteq T$ be a finite supporting subtree.  As $K$ is a finite union of finite arcs, the measure $\mu$ induces a globally defined Borel measure on $K$, which we denote by $\mu |_K$.  Let $\Gamma(K)$ denote the collection of all partial isometries of $K$ got by restricting the action $\FN \curvearrowright T$ to Borel subsets of $K$; the collection $\Gamma(K)$ forms a sort of pseudo-group, where the usual restriction that the domains of the partial maps be open has been dropped, and this pseudo-group is generated by maximal restrictions of elements of $\FN$ to $K$.  As $\mu$ is $\FN$-invaraint, $\mu |_K$ is $\Gamma(K)$-invariant.  The following is an exercise in definitions.

\begin{lemma}\label{L.LocalMeasures}
 Let $T \in \partial \cvn$, and let $K \subseteq T$ be a finite supporting subtree.  With notation as above, $\FN$-invariant measures on $T$ are in one-to-one correspondence with finite, positive $\Gamma(K)$-invariant Borel measures on $K$.
\end{lemma}

It follows from invariance of $\mu$ that replacing $K$ with a translate $gK$ induces a conjugacy between the pseudo-groups $\Gamma(K)$ and $\Gamma(gK)$.  A subset $X \subseteq T$ is \emph{measurable} if $X$ meets every translate of $K$ in a Borel set.  The set $X$ has \emph{measure zero} if $X$ meets every translate of $K$ in a set of $\mu |_K$-measure zero, and $X$ is \emph{full measure} if it complement in $T$ has measure zero.  The \emph{support} of $\mu |_K$, denoted $Supp(\mu |_K)$, is the complement in $K$ of the largest open $\mu |_K$-measure zero subset.  Note that $Supp(\mu |_K)$ contains an isolated point if and only if $\mu$ has an atom.  

If $\mu'$ is another measure on $T$, then say that $\mu'$ is \emph{absolutely continuous} with respect to $\mu$ if $\mu'|_K$ is absolutely continuous with respect to $\mu |_K$, and say that $\mu'$ is \emph{singular} with respect to $\mu$ if $\mu'|_K$ is singular with respect to $\mu |_K$.  Finally, say that $\mu'$ is \emph{dominated} by $\mu$, written $\mu' \leq \mu$, if for any Borel set $X \subseteq K$, one has $\mu'|_K(X) \leq \mu |_K(X)$.  Note that if $\mu' \leq \mu$, then $\mu'$ is absolutely continuous with respect to $\mu$; conversely, if $\mu'$ is absolutely continuous with respect to $\mu$, then $\mu'$ is homothetic to a measure that is dominated by $\mu$.

\subsection{Ergodic Components}

A measure $\mu \in \mathcal{M}(T)$ is called \emph{ergodic} if any $\FN$-invariant measurable subset $X \subseteq T$ either has full measure or has zero measure; this is the usual definition of ergodicity when translated to the pseudo-group $\Gamma(K)$.  Most of the time, we are only concerned with ergodic measures up to global rescaling.  

Fix a finite supporting subtree $K \subseteq T \in \partial \cvn$.  Presently, we will be concerned with measures $\mu \in \mathcal M(T)$ such that $Supp(\mu |_K)$ contains a non-degenerate arc; this feature is evidently independent of the choice of $K$, so we will simply say that $\mu$ has \emph{non-degenerate support} in this case.  We get the following very simple observation.

\begin{lemma}\label{L.NonDegenerateSupports}
 Let $T \in \partial \cvn$ have dense orbits, and let $K \subseteq T$ be a finite supporting subtree.  Suppose that $\mu, \mu' \in \mathcal{M}(T)$ are ergodic and have non-degenerate support.  If $Supp(\mu |_K) \cap Supp(\mu' |_K)$ is non-degenerate, then $Supp(\mu |_K)=Supp(\mu' |_K)$.  
\end{lemma}
  
\begin{proof}
 Let $I \subseteq Supp(\mu |_K) \cap Supp(\mu' |_K)$ be non-degenerate, then $\mu |_K(I)>0$ and $ \mu' |_K(I)>0$.  By ergodicity, $\Gamma(K)I$ is a full $\mu |_K$-measure (resp. $\mu' |_K$-measure) subset of $Supp(\mu |_K)$ (resp. $Supp(\mu' |_K)$).  To conclude note that supports are $\Gamma(K)$-invariant, so $\Gamma(K)I$ is contained in both supports and is evidently dense in both.  
\end{proof}

The Baire Category Theorem gives that $K$ is the union of supports of ergodic measures in $\mathcal{M}(T)$ with non-degenerate supports.  Along with Lemma \ref{L.NonDegenerateSupports}, we will use the following finiteness result of Guirardel from \cite{Gui98}.

\begin{lemma}\label{L.Gui}
 Let $T \in \partial \cvn$ have dense orbits.  The set $\mathcal{M}(T)$ contains at most $3N-4$ mutually non-homothetic ergodic measures.
\end{lemma}

Guirardel obtains Lemma \ref{L.Gui} only for non-atomic measures, but the statement above follows immediately by using Lemma \ref{L.AtomsFreeSplit} below and induction on rank.  Further, Guirardel's result can be extended to all trees in $\partial \cvn$ by restricting the class of measures allowed on simplicial edges to be atoms placed at the midpoint of the edge.  The importance of Lemma \ref{L.Gui} is that $T$ is finite dimensional from the current measure-theoretic point of view.  Note that Lemma \ref{L.Gui} allows us to express the ambient measure $\mu_{amb}$ on $T$ as a finite sum of mutually non-homothetic ergodic measures on $T$.

Put a relation $\sim$ on $\mathcal{M}(T)$; for $\mu, \mu' \in \mathcal{M}(T)$, say that $\mu \sim \mu'$ if $Supp(\mu |_K) = Supp(\mu' |_K)$.  It is easy to check that $\sim$ does not depend on $K$ (any two finite supporting subtrees are contained in a finite supporting subtree), and $\sim$ clearly is an equivalence relation; use $[\cdot]$ to denote the classes of $\sim$.   

\begin{cor}\label{C.Finiteness1}
 Let $T \in \partial \cvn$ have dense orbits.  There is a uniform bound on the number of orbits of subtrees in any transverse family in $T$.
\end{cor}

\begin{proof}
 Let $\mathscr{Y}$ be a transverse family in $T$, and let $k$ be the number of $[\cdot]$-classes of ergodic measures in $\mathcal{M}(T)$ with non-degenerate support; Lemma \ref{L.Gui} gives that $k \leq 3N-4$.  If $\mathscr{Y}$ contains more than $k$ orbits of trees, then Lemma \ref{L.NonDegenerateSupports} gives trees $Y,Y' \in \mathscr{Y}$ in different orbits, non-degerate arcs $I \subseteq Y$ and $I' \subseteq Y'$, and an ergodic measure $\mu \in \mathcal{M}(T)$ with non-degnerate support such that $\mu_I(I), \mu_{I'}(I') > 0$.  Choose a finite supporting subtree $K$ for $T$ that contains $I$ and $I'$; ergodicity gives an element of $\Gamma(K)$ taking a non-degenerate subarc of $I$ into $I'$.  Since $\mathscr{Y}$ is assumed to be a transverse family, this is impossible, so the number of orbits in any transverse family must be bounded by $k \leq 3N-4$.
\end{proof}

We also get:

\begin{cor}\label{C.TwoOrbits}
 Let $T \in \partial \cvn$ have dense orbits.  There is a transverse family in $T$ that contains more than one orbit if and only if there are measures $\mu, \mu' \in \mathcal{M}(T)$ with non-degenerate support such that $[\mu] \neq [\mu']$.
\end{cor}

\begin{proof}
 If $\mathscr{T}$ is a transverse family containing more than one orbit, then the proof of Corollary \ref{C.Finiteness1} gives that there are $\mu,\mu' \in \mathcal{M}(T)$ with non-degenerate support such that $[\mu] \neq [\mu']$.

 Conversely, let $\mu_1,\ldots,\mu_k$ be representatives of the $[\cdot]$-classes whose members have non-degnerate support; by assumption $k>1$.  By Lemma \ref{L.NonDegenerateSupports} there are arcs $I_1,\ldots, I_k \subseteq T$ such that no translate of $I_i$ intersects $I_j$ non-degenerately for $i \neq j$.  It follows that $\wedge_{i} \mathcal{X}_{I_i} = \emptyset$, so $\cup_i \mathcal{X}=\mathcal{X}_{I_i}$ is a transverse family by Lemma \ref{L.CommonRefinement}; evidently $\mathscr{X}$ contains $k$ orbits.
\end{proof}

\begin{prop}\label{P.Ergodic}
 Let $T \in \partial \cvn$ have dense orbits, and let $\mathscr{T}$ be a transverse family in $T$.  For any $Y \in \mathscr{T}$, $Stab(Y) \neq \{1\}$; further, the action $Stab(Y) \curvearrowright Y$ is not discrete.
\end{prop}

\begin{proof}
For any $Y \in \mathscr{T}$ and any non-degerate arc $I \subseteq Y$, there is an ergodic measure $\mu \in \mathcal{M}(T)$ with non-degenerate support such that $\mu_I(I) > 0$; by replacing $I$ with a non-degenerate sub-arc, we can assume that $Supp(\mu_I)=I$, so for any non-degenerate $J \subseteq I$, we have that $\mu_I(J) >0$.  Enlarge $I$ to a finite supporting subtree $K$ for $T$; clearly $Supp(\mu |_K) \supseteq I$.  By ergodicity of $\mu$, we have that the image of $J$ under $\Gamma(K)$ is a $\mu_I$-full measure subset of $I$.  So, choosing $J$ such that $I \ssm J$ contains a non-degenerate arc gives that there is $\gamma_g \in \Gamma(K)$, restricting $g \in \FN$, taking a non-degenerate sub-arc of $J$ onto a non-degenerate subarc $J' \subseteq I$.  It follows that $g \neq 1$.  Since $gY \cap Y \supsetneq J'$, and since $\mathscr{T}$ is a transverse family, we must have that $gY =Y$, hence $1 \neq g \in Stab(Y)$.  Choosing $J$ to be very short gives the statement that $Y$ is not discrete.
\end{proof}

Proposition \ref{P.Ergodic} implies that if $T \in \partial \cvn$ has a transverse family containing a simplicial edge if and only if $T$ does not have dense orbits.  

\subsection{Projections}

Fix $T \in \partial \cvn$, let $\mu \in \mathcal{M}(T)$ be non-atomic, and let $x, y \in T$.   There is a unique segment $[x,y] \subseteq T$ connecting $x$ to $y$.  Define a pseudo-metric $d'_{\mu}$ on $T$ by setting $d'_{\mu}(x,y):=\mu_{[x,y]}([x,y])$.  Invariance of $\mu$ gives invariance of $d'_{\mu}$.  The quotient of $T$ where $d'_{\mu}$ induces a metric is denoted $(T_{\mu},d_{\mu})$; we have a map $f_{\mu}:T \to T_{\mu}$.  Evidently, the map $f_{\mu}$ is Lipschitz exactly when $\mu$ is absolutely continuous with respect to $\mu_{amb}$.  

\begin{lemma}\label{L.Fibers}
 Let $T \in \partial \cvn$ have dense orbits, let $\mu \in \mathcal{M}(T)$ be absolutely continuous with respect to $\mu_{amb}$, and let $f_{\mu}:T \to T_{\mu}$ be the corresponding projection.  For any $x \in T_{\mu}$, $f_{\mu}^{-1}(x)$ is a closed subtree of $T$, and if $f_{\mu}^{-1}(x)$ is non-degenerate, then the collection of translates $\{gf_{\mu}^{-1}(x)\}_{g \in \FN}$ is a transverse family in $T$.  Further, $Stab(f_{\mu}^{-1}(x))=Stab(x)$; in particular, $Stab(f_{\mu}^{-1}(x))$ is finitely generated.
\end{lemma}

\begin{proof}
 Since $\mu$ is absolutely continuous with respect to $\mu_{amb}$, $\mu$ is non-atomic.  It is clear that for any $y \in T_{\mu}$, one has that $f_{\mu}^{-1}(\{y\})$ is path connected, hence a subtree of $T$.  For $z, x_n \in f_{\mu}^{-1}(\{y\})$ with $x_n$ converging to $x \in T$, we have that $\mu_{[z,x_n]}([z,x_n])=0$.  On the other hand, since $T$ is a tree, we have that $[z,x_n]$ converge uniformly to $[z,x]$, and since $\mu$ is non-atomic, we certainly have that $y \mapsto \mu_{[z,y]}([z,y])$ is continuous.  Hence, $\mu_{[z,x]}([z,x])=0$, so $f_{\mu}^{-1}(\{y\})$ is closed.  Finally, $\{gf_{\mu}^{-1}(\{y\})$ is a transverse family, since $f_{\mu}$ is an equivariant function, and the control on ranks is given by Lemma \ref{L.CanonicalPeripheral}; the statement about stabilizers is immediate from equivariance.
\end{proof}

%\subsection{Stabilizers of Elements in Certain Transverse Families}

%Let $T \in \partial \cvn$ have dense orbits, and let $\mathscr{Y}$ be a transverse family in $T$; assume that $\mathscr{Y}$ contains more than one orbits of subtrees.  By Corollary \ref{C.Finiteness1}, $\mathscr{Y}$ contains finitely many orbits of trees; let $\{Y_1,\ldots,Y_r\}$ contain exactly exactly one tree from each orbit, and let $\mathscr{Y}_i$ denote the subfamily of $\mathscr{Y}$ consisting of translates of $Y_i$.  The restriction $\mu_i$ of $\mu_{amb}$ to the orbit of $Y_i$ in $T$ is an invariant measure on $T$, and the measures $\mu_1,\ldots,\mu_r$ are mutually singular.  For any subset $S\subseteq \{1,\ldots,r\}$, there is a projection

%$$f_S=f_{\sum_{i \in S}\mu_i}:T \to T_S=T_{\sum_{i \in S}\mu_i}$$

%\noindent such that the set of non-degenerate point pre-images under $f_S$ give a transverse family $\mathscr{Y}_S$ in $T$ that is refined by $\mathscr{Y}_S=\cup_{i \in S} \mathscr{Y}_i$.  We get:

%\begin{lemma}\label{L.Finiteness2}
% Let $T \in \partial \cvn$ have dense orbits, and let $\mathscr{Y}$ be a transverse family in $T$.  If $\mathscr{Y}$ contains more than one orbit of trees, then for any $Y \in \mathscr{Y}$, $Stab(Y)$ is contained in a vertex group of a very small splitting of $\FN$.  
%\end{lemma}

Suppose that an invariant measure $\mu$ on $T$ has an atom, \emph{i.e.} $\mu |_K$ has an atom, say $a \in K$.  As $\mu |_K$ is a finite measure, the $\Gamma(K)$-orbit of $a$ is finite, and since $K$ is a supporting subtree for $T$, we have that for any finite subtree $K' \subseteq T$, the orbit of $a$ meets $K'$ in a finite subset.  In particular, for any hyperbolic element $g \in \FN$, the orbit of $a$ meets any fundamental domain of the axis of $g$ in a finite set.  If $s \in T$ is such that the orbit of $s$ meets every finite subtree of $T$ in a finite set, then we say that the orbit of $s$ is \emph{sparse}, or that $s$ is sparse.  So, if $T$ has an invariant measure with an atom, then $T$ contains a sparse point.  More generally, for a direction $d$ based at a point $x \in T$, if for any finite subtree $K' \subseteq T$, one has that the set $\{gx|gd \cap K'$ is non-degnerate$\}$ is finite, then we say that the orbit of $d$, or just $d$, is \emph{sparse}.  So, if $s \in T$ is sparse, then every direction based at $s$ is sparse.  

\begin{lemma}\label{L.AtomsFreeSplit}
 Let $T \in \partial \cvn$ have dense orbits.  If $T$ contains a sparse direction, then $T$ has a transverse covering $\mathscr{Y}$, such that for each $Y \in \mathscr{Y}$, $Stab(Y)$ is a proper free factor of $\FN$.  In particular, there is a uniform bound on the number of sparse directions in $T$ and hence a uniform bound on the number of sparse points in $T$.
\end{lemma}

\begin{proof}  

%We defined directions at $x$ to be components of $T \ssm \{x\}$, but recall that the set of directions at $x$ is identified with equivalence classes of non-degenerate segments $[x,y]$, where $[x,y]$ is equivalent to $[x,y']$ if and only if $[x,y] \cap [x,y']$ is non-degenerate.  Let $y \in K$ be such that $[x,y]$ represents the direction $d$.  Let $1x=g_0x, g_1x, \ldots,g_rx \in \mathcal{O}(x) \cap K$ be the list of all points of $gx \in \mathcal{O}(x) \cap K$ such that $g[x,y] \cap K = [gx,y_g]$ is non-degenerate; finiteness of this list follows from sparseness of $d$.  Choosing $y$ closer to $x$ (such that $[x,y]$ still represents $d$) ensures that $[g_ix,g_iy] \subseteq K$ for $i=0,\ldots,r$.  

 %We now blow-up $d$.  Let $\mathscr{X}$ be the collection of closures of components of $K \ssm \{g_ix\}$; let $x_i$ denote the copy of $g_ix$ contained in the element of $\mathscr{X}$ that does not contain $g_iy$, and let $x_i'$ denote the copy of $g_i$ contained in the element of $\mathscr{X}$ that contains THIS IS NOT QUITE RIGHT.....attach $r+1$ copies of $[0,1]$ to $\mathscr{X}$ 

 Let $d$ be a sparse direction based at $x \in T$.  We blow-up the direction $d$ to a finite segment: consider $X=T\ssm d \coprod \overline{d}$, and let $x_0 \in T$ and $x_1 \in \overline{d}$ denote the images of $x$ in $X$, and glue $[0,1]$ to $X$ by attaching $i$ to $x_i$.  Extend this operation equivariantly over $T$ to get $T'$, and let $\mathscr{I}=\{g[0,1]\}$ denote the collection of added-on arcs.  Since $d$ was assumed sparse, we have that $T'$ is an $\FN$-tree.  Put $\mathcal{X}$ to be the family of closures of components of $T' \ssm \cup \mathscr{I}$; every element of $\mathscr{X}$ is non-degenerate, so $\mathscr{X}$ is a transverse family.  Every element of $\mathscr{X}$ has non-trivial stabilizer by Proposition \ref{P.Ergodic}, and, evidently, $\mathscr{X} \cup \mathscr{I}$ is a transverse covering of $T'$.  

 Since $T$ has dense orbits, arc stabilizers in $T$ are trivial, which implies that $Stab(\{x\}) \cap Stab(d) = \{1\}$; hence, the stabilizer of any $I \in \mathscr{I}$ is trivial.  By construction, collapsing each element of $\mathscr{I}$ to a point gives an equivariant map $p:T' \to T$, and the image of $\mathscr{X}$ under $p$ is a transverse covering $\mathscr{Y}$ of $T$.  On the other hand, collapsing each element of $\mathscr{X}$ to a point produces an equivariant map $s:T' \to S$ to a Bass-Serre tree $S$ for a free splitting for $\FN$, which is easily seen to be the skeleton of $\mathscr{Y}$.  Hence, the stabilizer of any element of $\mathscr{Y}$ is a proper free factor of $\FN$.

\end{proof}

The proof gives a non-trivial, proper factor $F' \leq \FN$ that dynamically reduces $T$; further, $F'$ is canonical in the sense that for any factor $F''$ reducing $T$, one certainly has that $F' \cap F'' \neq \{1\}$, after conjugation.  Note that mixing actions have no sparse directions.

\section{Characterizing Arational Trees}

In this section, we obtain our first main result, a characterization of arational trees; before doing this, we bring a few results that will allow us to reduce to the case of a geometric tree.

\begin{lemma}\label{L.DegenerateResidual}
 Let $T \in \partial \cvn$ be geometric; let $K \subseteq T$ be a finite supporting subtree; and let $\mathscr{Y}$ be a transverse family for $T$.  Suppose that the residual of $\mathscr{Y}$ in $K$ is infinite and degenerate, then:
 \begin{enumerate}
  \item [(i)] $T$ does not have dense orbits,
  \item [(ii)] if $I \subseteq T$ is an arc such that $I \ssm \mathscr{Y}$ has infinitely many components, then $I$ intersects the simplicial part of $T$ non-degenerately, and   
  \item [(iii)] for any arc $I \subseteq T$ and any $Y \in \mathscr{Y}$, only finitely many translates of $Y$ meet $I$ non-degenerately; in particular, $\mathscr{Y}$ necessarily contains infinitely many orbits.  
 \end{enumerate}
 
\end{lemma}

\begin{proof}
 The contrapositions of statements (i) and (ii) follow immediately from Corollary \ref{C.GeomGraph}.  For statement (iii), since (ii) holds, we have by Corollary \ref{C.GeomGraph} that there is a (canonical) transverse family $\mathscr{X}$ consisting of the simplicial edges of $T$; set $\mathscr{Y}':= \mathscr{X} \wedge \mathscr{Y}$, then $\mathscr{Y}'$ satisfies the hypotheses of the statement as well, and $\mathscr{Y}'$ projects to a transverse family in the simplicial tree $T'$ got by collapsing all indecomposable subtrees of $T$ coming from Corollary \ref{C.GeomGraph}.  Now a transverse family in the simplicial tree $T'$ is equivalent to a subset $A \subsetneq T/\FN$ all of whose components are non-degenerate.  Statement (iii) immediately follows.
\end{proof}

\begin{prop}\label{P.NonGeomFamily}
 Let $T \in \partial \cvn$ have dense orbits, and let $\mathscr{Y}$ be a transverse family for $T$.  If there is an arc $I \subseteq T$ such that infinitely many elements of $\mathscr{Y}$ intersect $I$ non-degenerately, then for some $Y \in \mathscr{Y}$, $Stab(Y)$ is contained in a proper free factor of $\FN$.   
\end{prop}

\begin{proof}
 Choose $Y \in \mathscr{Y}$ to be such that any finite supporting subtree of $T$ meets the translates of $Y$ in infinitely many non-degenerate arcs.  Enumerate $Stab(Y)$ as $g_1,g_2,\ldots,g_i,\ldots$; put $H_k:=\langle g_1,\ldots, g_k\rangle$.  We will argue that there are proper free factors $F^1 \leq F^2\leq \ldots F^i \leq \ldots$, such that $H_k \leq F^k$; this sequence must stabilize by Lemma \ref{L.FactorProperties}.  

 Fix a basis $\mathcal{B}$ for $\FN$, and choose an invasion $K_1 \subseteq K_2 \subseteq \ldots \subseteq K_j \subseteq \ldots$ of $T$ by finite subtrees; let $\mathscr{K}_j$ denote the band complex associated to $(K_j,\mathcal{B})$, and let $f_j:T_j=T_{\mathscr{K}_j} \to T$ be the corresponding resolution.  We choose the $K_j$'s to ensure $Y \cap K_1$ is non-degenerate.  For each $k$, there is $j_k$ such that $g_i^{-1}K_l \cap K_l$ and $K_l \cap g_iK_l$ is non-empty for every $i \leq k$ and $l \geq j_k$.  This means that for $l \geq j_k$, the resolutions $f_l:T_l \to T$ restrict to give resolutions $f_l^{H_k}:(T_l)_{H_k} \to T_{H_k}$.  

 Recall that the maps $f_j$ are \emph{morphisms} of trees--every arc of $T_j$ can be subdivided into finitely many arcs such that $f_j$ is isometric on these smaller arcs; in particular, $f_j$ never maps a non-degenerate arc to a point.  The maps $f_j$ are also equivariant, hence the family $\mathscr{Y}_j$ of subtrees of $T_j$ given by $\mathscr{Y}_j:=\{f_j^{-1}(Y')|Y' \in \mathscr{Y}\}$ is a transverse family in $T_j$.  Note that for every $l \geq j_k$, every translate of $(T_l)_{H_k}$ is contained in some element of $\mathscr{Y}_l$.  Let $\mathscr{Y}_l'$ be the subfamily of $\mathscr{Y}_l$ consisting of those trees that contain a translate of $(T_l)_{H_k}$.  By Lemma \ref{L.DegenerateResidual} and our assumptions, there is a simplicial edge $e$ of $T_l$ such that $e \smf \mathscr{Y}_l$ contains a non-degenerate arc, say $e_0$.  Collapse the the components of the complement of the union of the translates of the interior of $e_0$.  The result is a non-trivial simplicial tree $T_l'$, where $H_k$ fixes a point.  As $T$ has dense orbits, arc stabilizers in $T$ are trivial, hence the same is true for $T_l$, so point stabilizers in $T_l'$ are proper free factors, \emph{i.e.} $H_k$ is contained in a proper factor.

 Let $F^k$ denote the smallest proper factor containing $H_k$; this exists by Lemma \ref{L.FactorProperties}.  Note for $l > k$, we have $H_k \leq H_l$, so $F^k \cap F^l$ contains $H_k$ and hence $F^k$, by definition of $F^k$.  Hence for $l > k$ $F_k \leq F_l$; therefore the sequence $F^1 \leq \ldots \leq F^k \leq \ldots $ eventually stabilizes with a proper factor $F'$ that contains every element of $Stab(Y)$.
\end{proof}

Proposition \ref{P.NonGeomFamily} will allow us to reduce our proof of Theorem \ref{T.Arational} to the case of geometric trees; the following result allows us to further reduce to the case of mixing trees.

\begin{prop}\label{P.NotMixing}
 Let $T \in \partial \cvn$ have dense orbits.  If $T$ is not mixing, then $T$ is dynamically reduced by a proper factor of $\FN$. 
\end{prop}

\begin{proof}
 Since $T$ is not mixing, there is a non-degnerate arc $I \subseteq T$, such that $X_I$ as in Discussion \ref{D.BuildTF} generates a transverse family $\mathscr{X}$ in $T$; we have that if $gX_I \neq X_I$, then $gX_I \cap X_I =\emptyset$.  According to Proposition \ref{P.NonGeomFamily}, if $\mathscr{X}$ meets some arc $J$ of $T$ in infinitely many non-degenerate segments, then $Stab(X_I)$ is contained in a proper factor of $\FN$.  On the other hand, Proposition \ref{P.Ergodic} gives that $X_I$ is not simplicial, hence there is a proper factor of $\FN$ acting non-simplicially on $T$, so Lemma \ref{L.FactorFamilies} and Corollary \ref{C.CarryToReduce} gives that $T$ is dynamically reduced by proper factor.  

 Hence, we assume that for any finite non-degenerate arc $J$ of $T$, $\mathscr{X}$ meets $J$ finitely many times.  Suppose first that $\mathscr{X}$ meet every arc of $T$.  Let $K$ be a finite supporting subtree of $T$.  Evidently, there are finitely many $g_1X_I, \ldots, g_rX_I$ such that $R:=K \ssm (g_1X_I \cup \ldots \cup g_rX_I)$ is a finite set; otherwise, there is an arc of $T$ met by $\mathscr{X}$ infinitely many times, and we are in the above case.  Consider the case $R=\emptyset$.  Since distinct elements of $\mathscr{X}$ are disjoint and since each element of $\mathscr{X}$ is a proper subtree of $T$, we get an arc $J \subseteq K$ and an element $gX_I$ such that $gX_I \cap K$ is not closed in $K$, giving that $X_I$ is not closed.  On the other hand, we have that $\overline{\mathscr{X}}$ is a transverse covering of $T$, and for any $x \in \overline{X_I} \ssm X_I$, we have that $Stab(x) \cap Stab(X_I)=\{1\}$.  Hence the skeleton of the corresponding graph of actions structure on $T$ has an edge with trivial stabilizer, and so a proper factor reduces $T$; see the proof of Lemma \ref{L.AtomsFreeSplit}.

 Now suppose that there is an arc $J \subseteq T$ that is not met by $\mathscr{X}$.  This gives that $X_J$ is a proper subtree of $T$ not meeting $\mathscr{X}$ such that the transverse family $\mathscr{Y}$ generated by $X_J$ is disjoint from $\mathscr{X}$, hence $\mathscr{Z}=\mathscr{X}\cup\mathscr{Y}$ is a transverse family.  Choose a basis for $\FN$ and an invasion $K_n$ of $T$ by finite subtrees; let $T_n$ be the geometric tree corresponding to $K_n$.  The preimage of $\mathscr{Z}$ under the resolving map is a transverse family $\mathscr{Z}_n$ in $T_n$, and $\mathscr{Z}_n$ contains at least two orbits of trees.  It follows from the proof of Proposition \ref{P.NonGeomFamily} that if some $Z \in \mathscr{Z}_n$ meets the simplicial part of $T_n$, then some $Z' \in \mathscr{Z}_n$ has stabilizer contained in a proper factor of $\FN$; this implies that the same holds for some $Z'' \in \mathscr{Z}$.  Hence, we reduce to the case that $T$ is geometric.

 Since $T$ has dense orbits, Corollary \ref{C.GeomGraph} gives that $T$ has a transverse covering $\mathscr{W}$ by indecomposable subtrees; since $\mathscr{Z}$ contains two orbits of trees, the same holds for $\mathscr{W}$.  Choose a band complex $\mathcal{X}$ to which $T$ is dual, then since $\mathscr{W}$ contains two orbits of trees, $\mathcal{X}$ contains at least two minimal components.  If some minimal component is thin, then some other component is contained in a proper factor.  Hence, we assume that every component is surface.  In this case, Lemma 4.1 (see also Corollary 5.2) of \cite{BF94} gives that some minimal component of $\mathcal{X}$ or some point stabilizer in $T$ is contained in a proper factor of $\FN$; in either case, we can conclude by Corollary \ref{C.CarryToReduce}.  
\end{proof}

If $T$ is simplicial, then Corollary \ref{C.Simp} (from the Appendix) gives that $T$ is not arational.  More generally, if $T$ does not have dense orbits, then Lemma \ref{L.GoodGraph} gives that $T$ is a graph of actions with vertex actions either simplicial or with dense orbits; further, the skeleton $S$ of the graph of actions structure on $T$ is very small.  By Corollary \ref{C.Simp}, every edge stabilizer in $S$ is contained in a proer free factor, and by Corollary \ref{C.CarryToReduce}, there is a proper factor reducing $T$.

%collapsing each simplicial vertex action to a point gives a non-trivial tree $T'$ with dense orbits, where the stabilizers of the vertex actions with dense orbits are contained in vertex stabilizers of $T'$.  Lemma \ref{L.Simp} gives that some vertex or edge stabilizer in $T'$ is a proper factor, hence either the stabilizer of some vertex action in $T$ with dense orbits is contained in a proper factor, or an edge stabilier in $T$ is a proper factor.  In either case, $T$ is not arational.  Hence, if $T$ is arational, then $T$ has dense orbits, and we are in Case $(i)$ of Theorem \ref{T.CanonicalFactors}.  

If $T$ is free, then $T$ is arational if and only if $T$ is indecomposable; indeed, Proposition \ref{P.Indecomp} gives that free and indecomposable implies arational.  On the other hand, Proposition \ref{P.NotMixing} gives that arational implies mixing, and Lemma \ref{L.MixGraph} gives that $T$ must be indecomposable--if $T$ is mixing but not indecomposable, then $T$ has a transverse covering, whose skeleton is a free splitting, since $T$ is free.  So, to have a characterization of arational trees, we need only understand non-free arational trees.

\begin{thm}\label{T.Arational}
 Let $T \in \partial \cvn$.  The following are equivalent.
 \begin{enumerate}
  \item [(i)] $T$ is arational,
  \item [(ii)] $T$ is indecomposable, and if $T$ is not free, then $T$ is dual to an arational measured foliation on a surface with exactly one boundary component.
 \end{enumerate}

\end{thm}

\begin{proof}
 We need only handle the case where $T$ is not free and is mixing, so we assume this.  First, we show that if $T$ is not geometric, then $T$ is not arational.  Take a resolution $f:T' \to T$ , with $T'$ geometric, in which a point stabilizer in $T$ fixes a point in $T'$; we want to see that the decomposition of $T'$ as in Lemma \ref{L.GoodGraph} has a simplicial component.  Toward contradiction, suppose that every component of $T'$ is minimal, hence $T'$ has dense orbits.  Since $T$ is not geometric, $f$ is not exact; since $f$ is a morphism of trees, this means that there is some arc $I \subseteq T'$ that is folded by $f$.  It is easy to see that in this case, $T$ cannot have trivial arc stabilizers, which contradicts $T$ having dense orbits.  Hence, $T'$ has a simplicial component.  Edge stabilizers in $T'$ must be trivial, since $T$ has trivial arc stabilizers, and we conclude that point stabilizers in $T'$ are contained in proper free factors, hence $T$ is not arational.

 Hence, we are left to understand the case that $T$ is geometric; suppose that $T$ is dual to a band complex $X$.  It is an easy exercise in definitions to check that mixing implies that $X$ has exactly one minimal component; see Lemma \ref{L.MixGraph} and Corollary \ref{C.GeomGraph}.  Now, Corollary \ref{C.GeomGraph}, Lemma \ref{L.MixGraph}, and Discussion \ref{D.BuildTF} give that if $X$ is not pure, then $T$ has a transverse covering; let $S$ denote its skeleton.  Let $Y$ be a representative for the transverse covering of $T$ by indecomposable trees; then $Y$ is geometric and dual to a pure minimal band complex $M \subseteq X$ \cite{Gui08}.  Rips Theorem gives that $M$ either is thin or surface type (toral is impossible here).  If $M$ is of thin type, then running the Rips machine on $X$ eventually produces a band that is disjoint from any loop contained in a leaf; it follows that every point stabilizer in $T$ is contained in a proper free factor of $\FN$, so $T$ is not arational by Corollary \ref{C.CarryToReduce}.

 Hence, $M$ must be of surface type, and up to homotopy, $M$ is a foliated surface, and $X$ is an adjunction space $M \coprod_f G$, where $G$ is a graph, and $f$ is a map from the boundary components of $M$ into $G$.  Since every boundary component of $M$ represents a cyclic subgroup of $\FN$, we get that $S$ must be very small, and Corollary \ref{C.Simp} allows us to conclude that $T$ is not arational.  
 
 Finally, we are in the case that $X$ is pure; in this case $T$ is dual to a surface $M$ carrying a minimal measured foliation, and the condition that $T$ is arational is equivalent, via well-know facts about surfaces, to $M$ having one boundary component (else all boundary components represent conjugacy classes of free factors).
\end{proof}

\section{Controlling Factor Reducing Systems}

The aim of this section is to provide control over all factor reducing systems for a given tree.  For $T \in \partial \cvn$ let $\mathcal{R}(T)$ denote the set of all factors reducing $T$; we will show:

\begin{thm}\label{T.Control}
 There is a tree $T' \in \partial \cvn$ such that for any $F \in \mathcal{R}(T)$, some element of $F$ fixes a point in $T'$.
\end{thm}

Say that a tree $T'$ as in the theorem \emph{controls} $T$; note that if every element of $\mathcal{R}(T)$ is peripheral, then $T$ controls $T$, so we need only consider the case where $T$ is dynamically reduced by some factor.  Put $\mathcal{D}(T)$ to be the set of all factors dynamically reducing $T$, and let $\mathcal{MD}(T) \subseteq \mathcal{D}(T)$ denote the set of minimal (with respect to inclusion) factors dynamically reducing $T$.  Let $\mathscr{T}_{\mathcal{MD}}$ denote the corresponding invariant family of sub-trees of $T$; note that each orbit in $\mathscr{T}_{\mathcal{MD}}$ is a transverse family.

\begin{prop}\label{P.MinTF}
 $\mathscr{T}_{\mathcal{MD}}$ is a transverse family.
\end{prop}

\begin{proof}
 Let $Y,Y' \in \mathscr{T}_{\mathcal{MD}}$, and suppose that $I \subseteq Y \cap Y'$ is an arc.  By Corollary \ref{C.FactorStab}, we have that $Stab(Y), Stab(Y') \in \mathcal{MD}(T)$.  Proposition \ref{P.Ergodic} gives that $Stab(Y \cap Y')$ contains a hyperbolic element and is not simplicial.  Minimality of $Y$ and $Y'$ give that $Stab(Y \cap Y') = Stab(Y) \cap Stab(Y')$, hence by Lemma \ref{L.FactorProperties} we have a factor $F \leq Stab(Y),Stab(Y')$ acting with dense orbits on its minimal tree, so by definition of $\mathscr{T}_{\mathcal{MD}}$, we must have $Stab(Y)=F=Stab(Y')$, hence $\mathscr{T}_{\mathcal{MD}}$ is a transverse family.
\end{proof}

We now prove Theorem \ref{T.Control}.

\begin{proof}
 Let $\mathscr{T}_{\mathcal{MD}}$ be the transverse family defined above.  Suppose that there is a non-degenerate arc $I \subseteq T$ that does not intersect any member of $\mathscr{T}_{\mathcal{MD}}$ non-degenerately.  This gives an ergodic measure $\mu \in \mathcal{M}(T)$ absolutely continuous with respect to $\mu_{amb}$ and with non-degenerate support such that $\mu(I') >0$ for some sub-arc $I' \subseteq I$, hence we have a projection $T \to T_{\mu}$.  Evidently, for any $Y \in \mathcal{MD}(T)$, we have that $Y$ fixes a point in $T_{\mu}$; whence, $T_{\mu}$ controls $T$.

We are left to handle the case that every arc of $T$ meets some element of $\mathscr{T}_{\mathcal{MD}}$ non-degenerately.  If there is $F \in \mathcal{MD}(T)$ such that $T_F$ is not arational, then we can also proceed as above: find $\mu \in \mathcal{M}(T)$ absolutely continuous with respect to $\mu_{amb}$ such that $\mu(I)>0$ for some $I \subseteq T_F$, then $T_{\mu}$ controls $T$.  So, we assume that for every $F \in \mathcal{MD}(T)$, $T_F$ is arational; by Theorem \ref{T.Arational}, all point stabilizers in $T_F$ are cyclic.  On the other hand, if any element of $Y \in \mathscr{T}_{\mathcal{MD}}$ is not free, then we find a measure $\mu$ as above with non-degenerate support contained in $Y$, and by Lemma \ref{L.NonDegenerateSupports} we have that $T_{\mu}$ controls $T$.  

Hence we assume that every element of $\mathscr{T}_{\mathcal{MD}}$ is free and arational (hence free and indecomposable).  If $\mathscr{T}_{\mathcal{MD}}$ is a transverse covering, then we are done, since trivial point stabilizers ensure that the skeleton $S$ is a free splitting of $\FN$; in this case $S$, with any choice of metric, controls $T$.  So, we assume that $\mathscr{T}_{\mathcal{MD}}$ is not a transverse covering.  

If $\mathscr{T}_{\mathcal{MD}}$ contains more than one orbit, collapse all but one orbits to get a tree $T'$ and a transverse family $\mathscr{T}_{\mathcal{MD}}'$ in $T'$ containing exactly one orbit of trees; this is possible by the aforementioned procedure, and since members of $\mathscr{T}_{\mathcal{MD}}$ are indecomposable, the elements of the non-collapsed orbit inject into $T'$.  Hence, we have that the members of $\mathscr{T}_{\mathcal{MD}}'$ are free and indecomposable.  

Let $K \subseteq T'$ be a finite supporting subtree.  If the residual $K \smf \mathscr{T}_{\mathcal{MD}}'$ is infinite, then $\mathscr{T}_{\mathcal{MD}}'$ meets $K$ in infinitely many disjoint, non-degenerate arcs.  In this case, we can argue as in the proof of Proposition \ref{P.NonGeomFamily} to get a (simplicial) tree $T''$ such that $Stab(Y)$, $Y \in \mathscr{T}_{\mathcal{MD}}'$ fixes a point in $T$ and such that every subgroup fixing a point in $T'$ also fixes a point in $T''$.  Hence, $T''$  controls $T$.  

So, suppose that the residual $K \smf \mathscr{T}_{\mathcal{MD}}'$ is finite; this implies in particular that $X=K \ssm \cup \mathscr{T}_{\mathcal{MD}}'$ is finite.  If $X \neq \emptyset$, then $K$ contains a finite $\Gamma(K)$-invariant set, hence $T'$ contains a sparse point; in this case Lemma \ref{L.AtomsFreeSplit} gives that $T'$ has a free splitting, and the skeleton of this  transverse covering controls $T$.  We are left to consider the case where $X = \emptyset$.  

If $\mathscr{T}_{\mathcal{MD}}'$ is a transverse covering of $T'$, then we are done by above, so we suppose not; note that if $\overline{\mathscr{T}}_{\mathcal{MD}}'$ is a transverse covering, then we also, done, hence we assume this is not the case.  Hence, we find an arc $I \subseteq K$ such that $I$ is not covered by finitely many elements of $\mathscr{T}_{\mathcal{MD}}'$.  On the other hand, $I$ is covered by $\mathscr{T}_{\mathcal{MD}}'$, so arguing again as in the proof of Proposition \ref{P.NonGeomFamily}, we find a (simplicial) $T''$ that controls $T$.
\end{proof}

\section{Nielsen-Thurston Classification for Elements of $Out(\FN)$}

As an application of our techniques, we deduce a variant of the Bestvina-Handel structure theorem for elements of $Out(\FN)$ obtained in \cite{BH92}; see also \cite{BFH97}.  All results in this section are known.  Our approach was started by Sela in \cite{Sela96} but was not completed, as Sela did not develop the requisite structure theory for non-geometric trees; in particular, Theorem 1.3 of \cite{Sela96} is false, as that statement critically depends on the fact that the group $G$ is freely indecomposable, which is ``maximally false'' for a free group (see examples in the Appendix).  

We follow the general idea of Sela's approach, to use dynamical invariants associated to a ``limit tree'' produced from an element $\Phi \in Out(\FN)$ to obtain invariants for $\Phi$.  The starting point is to build this (very small) limit tree from by $\Phi \in Out(\FN)$, where $\Phi$ is assumed here to have infinite order; finite order elements stabilize a tree in Outer space \cite{Cu84}.  We will only sketch the setup, since all this is well-known by now; see \cite{Bes88} and \cite{Sela96} for more details.  A main point to note is that nowhere do we need the train track machinery introduced in \cite{BH92}.

Let $T_0$ be a free, simplicial $\FN$-tree; for example, take $T_0$ to be a Cayley tree for $\FN$ whose edges have been identified with $[0,1]$.  Let $\phi \in Aut(\FN)$ be any lift of $\Phi$, and we consider the tree $T\phi$ got by twisting the action of $\FN$ on $T_0$ by $\phi$, \emph{i.e.} $g \in \FN$ acts on $T_0\phi$ as $\phi(g)$ acts of $T_0$.  Note that if $\iota \in Aut(\FN)$ is inner, then $T_0\iota$ is equivariantly isometric to $T_0$, so our set-up is independent of the choice of $\phi \in \Phi$.  Set $T_n:=T_0\phi^n$, then the image of the sequence $\{T_n\}$ in the compactified Outer space is subsequentially convergent to a (homothety class of) a very small tree $T$.  Further, $T$ comes with a bi-Lipschitz map $f:T \to T$ that satisfies, for $g \in \FN$ and $x \in T$, $f(gx)=\phi(g)f(x)$; see \cite{Sela96}.  Say that $f$ \emph{represents} $\phi$.  The following is immediate.

\begin{lemma}
 Let $T \in \partial \cvn$, and let $\Phi \in Out(\FN)$.  Suppose that there is a bi-Lipschitz map $f:T \to T$ that represents $\phi \in \Phi$.  If $\mathscr{Y}$ is a transverse family, then so is $f(\mathscr{Y})=\{f(Y)|Y \in \mathscr{Y}\}$.
\end{lemma}

As $\phi$ certainly preserves the set of free factors of $\FN$, to understand conjugacy classes of factors of $\FN$ preserved by $\Phi$, we need to find a subset of the collection of all factors reducing a tree $T$ that are guaranteed to be preserved by an $f$ as above.  We first handle trees that are dynamically reduced by some factor.

\begin{prop}\label{P.CanonDyn}
 Let $T \in \partial \cvn$; one of the following holds:
 \begin{enumerate}
  \item [(i)] For any proper factor $F \leq \FN$, $T_F$ is simplicial, or
  \item [(ii)] There are proper factors $F^1,\ldots,F^r$, each dynamically reducing $T$, such that:
  \begin{enumerate}
   \item [(a)] The union of the transverse families generated by the $F^i$'s is a transverse family $\mathscr{Y}$ in $T$,
   \item [(b)] If $f:T \to T$ is a bi-Lipschitz map representing $\alpha \in Aut(\FN)$, then for some $k$, $\alpha^k$ fixes each $F^i$, up to conjugacy.
  \end{enumerate}
 \end{enumerate}
\end{prop}

The subgroups $F^1,\ldots,F^r$ are called the \emph{characteristic dynamical factors} for $T$.

\begin{proof}
 If (i) does not hold, let $\mathscr{T}_{\mathcal{MD}}$ be the transverse family of minimal dynamical factors given by Proposition \ref{P.MinTF}; note that $\mathscr{T}_{\mathcal{MD}}$ contains finitely many orbits by Corollary \ref{C.Finiteness1}.  By the lemma, $f(\mathscr{T}_{\mathcal{MD}})$ is a transverse family, and since $f$ represents $\phi$, the stabilizers of element of $f(\mathscr{T}_{\mathcal{MD}})$ are factors; if some member of $f(\mathscr{T}_{\mathcal{MD}})$ is dynamically reduced by a factor $F$, then we immediately get that some element of $\mathscr{T}_{\mathcal{MD}}$ is dynamically reduced by $\alpha^{-1}(F)$, which is impossible.  Hence, $f(\mathscr{T}_{\mathcal{MD}})=\mathscr{T}_{\mathcal{MD}}$, and we take $\{F^1,\ldots,F^r\}$ to contain the stabilizer of one tree in each orbit of $\mathscr{T}_{\mathcal{MD}}$.  
\end{proof}

We also need to find canonical peripheral factors, in case $T$ is not dynamically reduced by any factor.

\begin{prop}\label{P.CanonPer}
 Let $T \in \partial \cvn$ and assume that $T$ is not dynamically reduced by any factor; one of the following holds:
 \begin{enumerate}
  \item [(i)] No non-trivial free factor fixes a point in $T$ ($T$ is arational), or
  \item [(ii)] There are factors $F^1,\ldots,F^r$, each peripherally reducing $T$, such that if $f:T \to T$ is a bi-Lipschitz map representing $\alpha \in Aut(\FN)$, then for some $k$, $\alpha$ fixes each $F^i$, up to conjugacy.
 \end{enumerate}

\end{prop}

For the proof, let $Fill(\cdot)$ denote the smallest free factor of $\FN$ containing $\cdot$; $Fill(\cdot)$ is well-defined by Lemma \ref{L.FactorProperties}.  The factors in the conclusion are called the \emph{characteristic peripheral factors} of $T$.

\begin{proof}
 The collection of point stabilizers in $T$ that are free factors is obviously $\alpha$-invariant, and Lemma \ref{L.CanonicalPeripheral} gives that this set has a finite number of conjugacy representatives.  Hence we consider $x \in T$ such that $Stab(x)$ contains a factor but is not a factor.  Since $T$ is not dynamically reduced by a factor, we get for any $g \in Stab(x)$ that either $Fill(g) \leq Stab(x)$ or $Fill(g)=\FN$, and the latter situation must occur, since $Stab(x)$ is assumed not to be a free factor.  

 To finish, we use resolutions of $T$ be geometric trees; since these arguments are very similar to ones given above, we will not provide all the details.  Take a geometric tree $T'$ such that every point stabilizer fixes a point in $T'$.  Note that if $T'$ contains a simplicial edge, then $T'$ must be simplicial; else there $T'$, hence $T$, will be dynamically reducible.  On the other hand, if $T'$ is simplicial, then every point stabilizer is a factor, which we assumed not to be the case.  Hence, $T'$ contains no simplicial edge, hence has dense orbits, so $T$ is geometric.

 Now, further note that $T$ could not contain a thin component, since in this case, point stabilizers cannot fill the whole group.  It follows that every minimal component of $T$ is a surface; hence $T$ has a very small splitting, whose skeleton is evidently preserved by any bi-Lipschitz map representing some $\alpha \in Aut(\FN)$; apply Corollary \ref{C.Simp} to find an invariant factor.
\end{proof}

\begin{thm}\label{T.Classify}
 Let $\Phi \in Out(\FN)$.  One of the following holds:
  \begin{enumerate}
   \item [(i)] $\Phi$ is finite order,
   \item [(ii)] there is $k$, such that $\Phi^k$ fixes a conjugacy class of non-trivial proper free factors of $\FN$, or
   \item [(iii)] $\Phi$ preserves a pair $(L_{\Phi}^-,L_{\Phi}^+)$ of minimal and filling laminations.
  \end{enumerate}

\end{thm}

\begin{proof}
 Assume that $\Phi$ is not finite order, and note that if $\Phi$ preserves a factor $F$, then $F$ carries a leaf of any limit tree; this uses compactness of spaces of currents and continuity of the Kapovich-Lustig intersection pairing.  Applying Propositions \ref{P.CanonDyn} and \ref{P.CanonPer}, we get the converse: if a limit tree of $\Phi$ is not arational, then $\Phi$ preserves a factor.  Hence, some limit tree of $\Phi$ is arational if and only if every limit tree for $\Phi$ is arational.  The laminations in (iii) are $L(T^+)$ and $L(T^-)$, where $T^+$ is a limit tree for $\Phi$ and $T^-$ is a limit tree for $\Phi^{-1}$.

\end{proof}

\section{Appendix: Examples of Trees}

We collect several examples of trees in $\partial \cvn$ as well as a few basic (known) results to be used in the main body of the article.  Aside from developing intuition, we feel that including these examples will give the reader insight into what is really being addressed in the (technical) arguments in this article.

\subsection{Bass-Serre Trees}  Let $\mathcal{G}$ be any graph of groups decomposition of $\FN$, and let $S$ be the corresponding Bass-Serre tree.  Let $\{e_1, e_2,\ldots \}$ be a set containing one edge in each orbit of edges in $S$.  Identifying $e_i$ with a segment in $\mathbb{R}$ and extending this operation equivariantly over $S$ gives a metric simplicial $\mathbb{R}$-tree $T$.  The definition of very small uses no metric data, so it makes sense to say that $S$ is very small.  If $S$ is very small then so is $T$; more generally, if each $e_i$ is identified with a non-degnerate segment of $\mathbb{R}$, then $S$ is very small if and only if $T$ is very small.

\subsection{Foldings}  Let $T$ be a metric simplicial tree with trivial edge stabilizers; in other words, the corresponding Bass-Serre tree encodes a free splitting of $\FN$.  Suppose that $T$ contains a vertex $v$ with non-trivial stabilizer.  Let $e=(v,v')$ be an edge incident on $v$, and let $1 \neq g \in Stab(v)$.  Since $\FN$ acts isometrically on $T$, and since $e$ is identified with a ($Isom(\mathbb{R})$ orbit of a) segment $[a,b] \subseteq \mathbb{R}$, we have that $ge=(v,gv')$ is isometric with $[a,b]$ as well.  Identify $e$ with $ge$ and extend this operation equivariantly over $T$ to get a new tree $T'=T_{e=ge}$; this is called \emph{folding} at $v$.  One has that $T'$ is very small if and only if $\langle g \rangle$ is a maximal cyclic subgroup of $\FN$.  

We now show that all very small simplicial trees arise from a free splitting trees via iteratively applying the above folding procedure.  The structure of cyclic splittings of $\FN$ is completely understood; we use the simple topological characterization due to Bestvina-Feighn: 

\begin{lemma}\cite[Lemma 4.1]{BF94}
 Let $\Gamma$ be a finite graph, let $S$ be a compact surface, and suppose that $f:\partial S \to \Gamma$ is a map that is essential on each component.  If the quotient space $\Gamma \coprod_f S$ has free fundamental group, then there is a homotopy equivalence $\psi:\Gamma \to S^1 \vee \Gamma'$, such that $\psi \circ f$ sends one component of $\partial S$ homeomorphically onto $S^1$ and sends all other components into $\Gamma'$.  
\end{lemma}

Now let $T$ be a simplicial very small tree, and let $Y$ be the corresponding Bass-Serre tree.  The graph of groups structure on $\FN$ coming from $Y$ corresponds to writing $\FN$ as the fundamental group of $S \coprod_f \Gamma$, where $\Gamma$ is a finite graph, and $S$ is a disjoint union of closed annuli; the components of $\Gamma$ are in one-to-one correspondence with orbits of vertices in $Y$.  

Apply the Lemma to get a homotopy equivalence $\psi:\Gamma \to \Gamma' \vee S^1$; let $e=(v,v')$ be the edge of $S$ corresponding to the annulus $A \subseteq S$ for which $\psi \circ f$ maps one component homeomorphically onto $S^1$.  This means that $\pi_1(A)=\langle a \rangle = Stab(e)=Stab(v) \cap Stab(v')$ is a free factor of one of $Stab(v)$, $Stab(v')$, say $Stab(v')$.  Suppose that $A=S^1 \times [-1,1]$ and that the restriction of $f$ to $S^1 \times \{1\}$ is a homeomorphism.  Apply a homotopy to $S \coprod_f \Gamma$ that collapses $A$ onto $S^1 \times \{-1\} \cup \{pt.\} \times [-1,1]$.  The induced effect on $Y$ is to unfold at the vertex $v$; more precisely, we get a tree $Y'$ and vertices $v, v'$ in $Y$ with $\langle a \rangle \leq Stab(v)$, such that $Y'_{(v,v')=a(v,v')}$ is equivarianlty isometric with $Y$.  Note that there are fewer orbits of edges with non-trivial stabilizer in $Y'$ than in $Y$, so our claim that all simplicial trees in $\partial \cvn$ arise from iteratively folding a metric Bass-Serre tree for a free splitting of $\FN$ follows by induction.

\begin{cor}\label{C.Simp}
 Let $T \in \partial \cvn$ be simplicial.  Every edge stabilizer in $T$ is contained in a proper free factor, and there is at most one conjugacy class of vertex stabilizers in $T$ that is not contained in a proper free factor.  
\end{cor}

\begin{proof}
 We have that $T$ unfolds to $T'$, where $T'$ has an edge $e$ with trivial stabilizer and such that every arc stabilizer in $T$ is contained in a vertex stabilizer of $T'$; collapsing every edge outside of the orbit of $e$ to a point gives the conclusions.  
\end{proof}

\subsection{Surface Trees}

References for this subsection are \cite{CB88}, \cite{BF95}, and \cite{MS84}.  Let $\Sigma$ be a hyperbolic surface, and let $\Lambda=(\lambda, \mu)$ be a measured lamination on $\Sigma$, where $\mu$ is assumed to have full support.  In this case $\lambda$ decomposes as a disjoint union of simple closed curve components $c_j$ and minimal components $m_i$ containing more than one leaf.  The dual tree $T_{\Lambda}$ is defined as follows: $\Lambda$ lifts to $\tilde{\Lambda}$ on the universal cover $\tilde{\Sigma}$ of $\Sigma$; now collapse each leaf of $\tilde{\lambda}$ and each complementary region of $\tilde{\lambda}$ to a point.  The resulting leaf space is a union of arcs $I$ coming from arcs $\tilde{I}$ in $\tilde{\Sigma}$ transverse to $\tilde{\lambda}$, and any two points in this space are contained in such an arc; $I$ gets a Borel measure from $\tilde{\mu}$, which gives a pseudo-metric $d_{\mu}$ on the leaf space: $d_{\mu}(x,y)=\inf \mu(I)$, such that $x$ and $y$ are the endpoints of $I$.  The quotient metric space is $T_{\Lambda}$.  The action of $\pi_1(\Sigma)$ on $\tilde{\Sigma}$ by deck transformations descends to an action of $T_{\Lambda}$, and invariance of $\tilde{\mu}$ gives that the induced action is by isometries.  

The following are easy exercises: (1) $T_{\Lambda}$ has dense orbits if and only if $\lambda$ contains no simple closed curve component, and (2) $\lambda$ contains a simple closed curve component if and only if $T_{\Lambda}$ contains a non-degnerate segment that generates a transverse family if and only if $T_{\Lambda}$ contains a non-degenerate segment with non-trivial stabilizer.  

The decomposition $\lambda = \cup_j c_j \bigcup \cup_i m_i$ gives a transverse covering of $T_{\Lambda}$ by subtrees that either are simplicial edges with $\mathbb{Z}$-stabilizer or indecomposable trees; indeed, indecomposable trees are defined precisely to generalize trees dual to minimal and filling surface laminations.

We now discuss the characteristic factors associated to a surface tree.  As a concrete example, consider a surface $\Sigma$ with one boundary component and genus two; we have an identification $\pi_1(\Sigma)=\FN$.  We explore two examples of laminations on $\Sigma$.  For discussion, lay $\Sigma$ flat on a table with the boundary component on the right, and think of $\Sigma$ as a punctured torus $S_l$ on the left, glued to a twice-punctured torus $S_r$ on the right.

{\bf 1:} Equip $\Sigma$ with a minimal measured lamination $(L,\mu)$, where $L$ fills up $S_l$ but avoids $S_r$.  In this case, it is well-known that $F=\pi_1(S_l)$ is a free factor of $\FN=\pi_1(\Sigma)$.  Now $F$ acts arationally on its minimal subtree in $T=T_{\Lambda}$, so $F$ is a characteristic dynamical factor for $T$.  The point stabilizers in $T$ are conjugate to $\pi_1(S_r)$, and the only elliptic element of $T$ that is contained in a proper free factor of $\FN$ is the left boundary component of $S_r$, so $T$ is not peripherally reduced by any factor.  

{\bf 2:} Equip $\Sigma$ with a minimal measured lamination $(L,\mu)$, where $L$ fills up $S_r$ but avoids $S_l$; let $T$ denote the corresponding tree.  Note that for any subgroup $H \leq \FN$ to dynamically reduce $T$, we certainly must have that $H$ carries a leaf of $L$, and since no leaf of $L$ is contained in a subsurface whose fundamental group is contained in a proper free factor of $\FN$, $T$ is not dynamically reducible by any factor.  On the other hand, since $F=\pi_1(S_l)$ is a free factor and evidently fixes a point in $T$, we have that $F$ peripherally reduces $T$.  Further, $F$ is the unique characteristic factor for $T$.

\subsection{Graphs of Actions and Extensions}

We recall an alternative point of view of transverse coverings that is convenient for constructions \cite{Lev94, Gui08}.  A \emph{graph of actions} $\mathscr{G}=(S, \{Y_v\}_{v \in V(S)}, \{p_e\}_{e \in E(S)})$ consists of:

\begin{enumerate}
 \item [(i)] a non-trivial simplicial tree $S$, called the \emph{skeleton}, equipped with an action (without inversions) of $\FN$,
 \item [(ii)] for each vertex $v \in V(S)$ of $S$ a tree $Y_v$, called a \emph{vertex tree}, and
 \item [(iii)] for each oriented edge $e \in E(S)$ with terminal vertex $v \in V(S)$ a point $p_e \in Y_v$, called an \emph{attaching point}.
\end{enumerate}

It is required that the projection sending $Y_v \rightarrow p_e$ is equivariant and that for $g \in \FN$, one has $gp_e=p_{ge}$.  Associated to a graph of actions $\mathscr{G}$ is an action of $\FN$ on a tree $T_{\mathscr{G}}$: define a pseudo-metric $d$ on $\coprod_{v\in V(S)} Y_v$: if $x \in Y_u$, $y \in Y_v$, let $e_1...e_k$ be the reduced edge-path from $u$ to $v$ in $S$, \emph{i.e.} $\iota(e_1)=u$, $\tau(e_k)=v$, and $\tau(e_i)=\iota(e_{i+1})$, then

$$
d(x,y)=d_{Y_u}(x,p_{\overline{e_1}}) + d_{Y_{\tau(e_1)}}(p_{e_1},p_{\overline{e_2}}) + ... +d_{Y_v}(p_{e_r},y)
$$

\noindent Passing to the usual quotient of this pseudo-metric space gives a metric space that is a tree, called the \emph{dual} of $\mathscr{G}$; we denote it by $T_{\mathscr{G}}$.  If $T$ is equivariantly isometric to some $T_{\mathscr{G}}$, then we say that $T$ \emph{splits} as a graph of actions.  Graphs of actions were defined and first explored in \cite{Lev94}; later, in \cite{Gui08}, transverse coverings were defined, and the following translation was noted.

\begin{lemma}\label{L.Graphs}\cite[Lemma 1.5]{Gui08}
 Assume that $T$ splits as a graph of actions with vertex trees $\{Y_v\}_{v \in V(S)}$, then the subset of $\{Y_v\}_{v \in V(S)}$ consisting of non-degenerate trees is a transverse covering for $T$.  Conversely, if $T$ has a transverse covering $\{Y_v\}_{v \in V}$, then $T$ splits as a graph of actions whose non-degenerate vertex trees are $\{Y_v\}_{v \in V}$.
\end{lemma}  

The skeleton $S$ is a splitting of $\FN$ coded by the transverse covering.  The degenerate vertex trees in a graph of actions structure record embedding information for point stabilizers in non-degnerate vertex trees.  
 
Now, we use graphs of actions to build new actions.  Let's first see how to build examples {\bf 1} and {\bf 2} from the previous section.  For {\bf 1} start with $S_l$ equipped with the lamination as above; this gives a one vertex action.  Note that $S_r$ deformation retracts onto a graph $G$; $\pi_1(G) \curvearrowright \{pt.\}$ is the other vertex action.  Let $S$ be the simplicial tree corresponding to splitting the surface $\Sigma$ along the boundary of $S_l$.  The image of the obvious (homotopy class of) map from the boundary of $S_l$ into $G$ lands in a free factor of $\pi_1(G)$.  All this data describes a graph of actions: attaching points are points with non-trivial stabilizer in the tree dual to the lamination on $S_l$, and everything is forced by $\FN$-invariance; the resulting tree is $T$ from {\bf 1}.  We similarly obtain the tree from {\bf 2}. 

We already know from Lemma \ref{L.GoodGraph} how to build trees that do not have dense orbits from vertex actions that are simplicial or have dense orbits, so we focus on actions with dense orbits.  For concreteness, let $T$ be an arational surface tree dual to $(L,\mu)$ on the surface $\Sigma$.  

\vspace{.2cm}

\noindent {\bf Basic Extension:}  For any two points $x,y \in T$, we can form an extension of $T$, which will be a vertex action, as follows.  Let $S$ be the Bass-Serre tree for the HNN-extension $\FN \ast_{\{1\}}$, place $T$ at vertex $v \in S$, let $e$ be an edge incident on $v$, and let $v'$ be the other vertex of $e$.  Glue $x \in T$ to the copy $y'$ of $y$ in the copy $T'$ of $T$ sitting at $v'$, and extend this operation equivariantly over $S$ to get a graph of actions, with corresponding $F_{N+1}$-tree $T'$.  

There are two slightly different possibilities for $T'$.  First, suppose that $x$ and $y$ are in the same orbit in $T$, and let $g \in \FN$ such that $gx=y$.  Let $t$ be the element of $F_{N+1}$ that sends $v$ to $v'$.  One sees that $x=y'=ty=tgx$, so we have introduced a point stabilizer in $T'$; this happens if and only if $x$ and $y$ are in the same orbit in $T$, as is easily checked.  If $x$ and $y$ are in different orbits, then orbits of points with non-trivial stabilizer in $T'$ are in one-to-one correspondence with orbits of points with non-trivial stabilizer in $T$.  The tree $T'$ is easily visualized: represent $x$ and $y$ by points in $L \subseteq \Sigma$, and attach the ends of a string to $x$ and $y$; think of the string as part of the lamination.  Now, pass to the universal cover of $\Sigma \cup string$ and collapse leaves to get $T'$.  The image of $\FN$ in $F_{N+1}$ is the unique characteristic factor, and $T'$ is mixing.

\vspace{.2cm}

\noindent {\bf Extension over Completion Points:} Gilbert Levitt appears to be the first to point out the following: if $T$ is a minimal (non-trivial) $G$-tree with dense orbits, such that $G$ is countable, and such that $T$ contains a branch point, then $T$ is not complete.  Here is his argument.  Since orbits are dense, branch points are dense; this implies that every arc in $T$ is nowhere dense.  Since $G$ is countable and since $T$ is minimal, $T$ is a countable union of arcs (fundamental domains of axes of hyperbolic elements).  The Baire Category Theorem gives that $T$ is not complete; use $\overline{T}$ to denote the metric completion of $T$.

Let $x \in \overline{T} \ssm T$, and let $y \in T$; note that the (unique) direction $d$ at $x$ is sparse in $\overline{T}$, as it meets an arc at most twice.  For an extension $T'$ of $T$ using $x$ and $y$ as above; this is not as easy to visualize as before, but morally looks very similar.  The $T$ gives a transverse family in $T'$ whose members are not closed subtrees of $T'$; further, one sees that $d$ remains a sparse direction in $T'$.  The tree $T'$ is not mixing, since no arc in $T$ can be translated to cover an arc in $T'$ containing a translate of $d$; on the other hand, $T'$ does satisfy a weaker mixing condition: every orbit is dense in every arc of $T'$.  The residual of any transverse family in $T'$ in any supporting subtree is finite (or empty).  The characteristic factor for $T'$ is the image of $\FN$ in $F_{N+1}$.

Now let $x,y \in \overline{T} \ssm T$; note that the directions $d_x$ at $x$ and $d_y$ at $y$ are both sparse.  Let $T'$ be the extension of $T$ as above.  Again $T$ generates a transverse family in $T'$ whose members are not closed.  Now the points $x$ (and $y$) are sparse in $T'$; further, $T'$ neither is mixing nor satisfies arc-dense orbits, since the orbit of $x$ does not meet any arc of $T$.  On the other hand, we retain that every transverse family has finite residual in every finite supporting subtree for $T'$, and the characteristic factor is again $\FN \leq F_{N+1}$. 

All of these extension constructions can be performed using more general skeleta, \emph{e.g.} for $S$ have more than one orbit of vertices; one just starts with a larger collection of vertex actions.  If we take $S$ to be a free splitting with quotient a segment, say, and with ranks of vertex stabilizers $N$ and $M$, and if we take arational, say, $\FN$-tree $T_1$ and arational $F_M$-tree $T_2$, then the resulting $T'$ will have two characteristic dynamical factors: the images of $F_N$ and $F_M$.

\vspace{.2cm}

\noindent {\bf Iterated Extensions:}  Let $T$ be an extension over completion points of an arational $\FN$-tree $T_0$, and choose a completion point $x \in T$ such that for any $z \in T_0$, the arc $[z,x]$ crosses infinitely many translates of $T_0$; that such points $z$ exist is an easy exercise.  Choose $y\in T$ to be some other point, and let $T'$ be the extension of $T$ over $x$ and $y$ as above.  Then the directions in $T$ that are sparse are no longer sparse in $T'$; however, $T'$ does contain a sparse direction as above.  Infinitely many distinct elements of the transverse family $\mathscr{Y}$ generated by $T_0$ meet any finite supporting subtree for $T'$ in an arc.  The tree $T'$ does not satisfy very nice mixing properties; however, it is the case that there is exactly one $[\cdot]$-class of ergodic measures with non-degenerate support, as is the case for every ``HNN-extension'' example presented so far.  Finally, the tree $T$ is reduced by two factors: $\FN$ and $F_{N+1}$, with $\FN$ being the characteristic factor.  

Of course, one can iterate all the above procedures starting with, or using at any step, Bass-Serre trees with more than one orbit of vertices, and possibly with non-trivial edge stabilizers.  In this case, it is rather transparent what will be the reducing factors and characteristic factors.  One obtains trees with more than one $[\cdot]$-class of ergodic measures by using, at some stage, a skeleton with more than one orbit of vertices, or by starting the construction with a tree containing several $[\cdot]$-classes, \emph{e.g.} a surface tree dual to a measured lamination with several (non-curve) minimal components.  

%\begin{lemma}
% Let $T \in \partial \cvn$.  If $T$ splits as a graph of actions with skeleton $S$ such that every edge stabilizer in $S$ is cyclic, then $S$ is a very small tree.
%\end{lemma}

%\begin{proof}
% Let $e$ be an edge of $S$ with $Stab(e)=\langle g \rangle \neq \{1\}$, and let $p_1 =Y_1$ and $p_2 \in Y_2$ denote the attaching points in the subtrees corresponding to the endpoints of $e$, so $Y_1$ is a point, and $Y_2$ is a non-degenerate subtree of $T$.  Note that $Stab(e)=Stab(Y_2) \cap Stab(p_1)$.  Let $h \in Stab(p_1)$ be non-trivial such that $h^k \in Stab(e)$
%\end{proof}

%It will be interesting to distinguish another particular kind of transverse family.  Say that a subtree $Y \subseteq T$ is (dynamically) \emph{malnormal} if for any $g \in \FN$, either $gY=Y$ or $gY \cap Y =\emptyset$; a transverse family all of whose elements are malnormal is called a \emph{malnormal family}.  The following gives a relationship between malnormality in the usual sense and dyanmical malnormality.  

%\begin{lemma}
% Let $T \in \partial \cvn$ have dense orbits, and let $Y, Y' \subseteq T$ be subtrees; then $Stab(Y \cap Y') = Stab(Y) \cap Stab(Y')$.   MAY NEED MORE HYPOTHESES...REASONABLE TO ASSUME THAT ANY STABY INVARIANT SUBTREE OF Y IS DENSE IN Y...	
%\end{lemma}

%\begin{proof}
% One has that $Y$ contains the axis of a hyperbolic element in $Stab(Y)$ and the (unique) fixed point of an elliptic element in $Stab(Y)$, so $Y \cap Y'$ contains a subtree that is invariant under $Stab(Y) \cap Stab(Y')$.  CONTINUE...
%\end{proof}

\subsection{Nesting, Interesting Residuals}

We now describe the most interesting basic examples of reducible trees.  Consider the following automorphism $\alpha$ of $F_6=F(a,b,c,d,e,f)$:
\begin{align*}
 a &\mapsto ab\\
 b &\mapsto baeb\\
 c &\mapsto cd\\
 d &\mapsto dced\\
 e &\mapsto ef^2ef\\
 f &\mapsto (fe)^2f^2ef^2ef
\end{align*}
Let $T_0$ be the Cayley tree of $F_6$ relative to the basis given, and metrize $T_0$ by identifying each edge with $[0,1]$.  The aim is to understand the limit tree got by iterating $\alpha$ on $T_0$; for this, it is helpful to use the obvious homotopy equivalence $f:T_0/\F_6 \to T_0/F_6$ that induces $\alpha$, which is a \emph{relative train track} map for $\alpha$.  Since the machinery is well-known at this point, we will use the language of relative train tracks for this example; see \cite{BH92}.  

The main point is that the bottom stratum, corresponding to $\{e,f\}$ is faster growing than the upper strata.  Color the bottom stratum green and the two ``sides'' of the top stratum red and blue, and start iterating the lift of $f$ representing $\alpha$ on $T_0$; after iterating for some time, rescale so that the maximal translation length of a generator is one.  After many iterations, one notes that the red and blue parts of the tree begin to resemble Cantor sets, while the there are large green subtrees.  In fact, the green part of the tree is a transverse family, which is also the case for the red and blue parts, but they are ``becoming degenerate'' after many iterations.  

Let $T$ be the corresponding limit tree.  We have the following transverse families in $T$: (1) the green transverse family $\mathscr{T}_g$, whose stabilizer is $\langle e,f \rangle$, (2) the green/red transverse family, whose stabilizer is $\langle a,b,e,f \rangle$, and (3) the green/blue transverse family, whose stabilizer is $\langle c,d,e,f \rangle$; the collection of these three subgroups is $\mathcal{R}(T)$.  

The tree is not indecomposable and does not split as a graph of actions; the residual of every transverse family in $T$ is degenerate and non-empty.  There are measures $\mu_{red},\mu_{blue} \in \mathcal{M}(T)$, which are supported on the corresponding colored sets.  The trees $T_{\mu_{red}+\mu_{blue}}$, $T_{\mu_{red}}$, and $T_{\mu_{blue}}$ all split as graphs of actions; $T_{\mu_{red}+\mu_{blue}}$ has two non-homothetic ergodic measures with non-degenerate support.

\subsection{Rigid Extnsions of Indecomposable Trees}

We finally give an example of a tree that splits as a graph of actions in an interesting way and should be thought of as a sort of rigid extension.  This example is very important for intuition; the idea for it came from a conversation with Vincent Guirardel and Gilbert Levitt a couple of years ago, and we thank them for explaining this point to us.  

Consider the following automorphism $\alpha$ of $F_4=F(a,b,c,d)$:
\begin{align*}
 a &\mapsto ab\\
 b &\mapsto bacb\\
 c &\mapsto d\\
 d &\mapsto cd
\end{align*}
Let $T_0$ be the Cayley tree for $F_4$ relative to the basis given, and metrize $T_0$ by identifying each edge with $[0,1]$.  As in the prior example, let $T$ be the limit got by iterating $\alpha$ on $T_0$.  Since the lower stratum is slower-growing than the top stratum, $T$ contains a point $x_0$ that is stabilized by $F_2=F(c,d)$.  We note that by construction, if (1) $(X,Y) \in L(T)$,  (2) if the first letter of $X$ is not the same as the first letter of $Y$, and (3) if $X$ contains letter from $\{a^{\pm},b^{\pm}\}$, then the closure of $\FN(X,Y)$ contains a minimal lamination $L_0$ contained in $\partial^2 F(c,d)$, namely the Bestvina-Feighn-Handel \emph{stable lamination} for the inverse of the lower stratum.  In particular, no such leaf of $L(T)$ is carried by a finitely generated subgroup of infinite index in $\FN$.

Using the main result of \cite{R10a}, one gets that $T$ is indecomposable.  We now form an extension of $T$.  Consider vertex actions $T$ and $F(e,f) \curvearrowright \{pt.\}$, where the attaching points of $T$ are translates of $x_0$.  Let $S$ be the Bass-Serre tree for a splitting $F(a,b,c,d)\ast_{F(\alpha,\beta)}F(e,f)$, where the $(\alpha,\beta)$ is sent to $(c,d)$ in the left factor, and $(\alpha,\beta)$ are sent to some, random say, rank-$2$ subgroup $H \leq F(e,f)$.  The main point is that $H$ should not be contained in a proper factor of $F(e,f)$.  Let $T'$ be the tree corresponding to this graph of actions; so $T'$ has a transverse covering by translates of $T$, whose stabilizers are conjugates of the image of $F(a,b,c,d)$, \emph{i.e.} $F(a,b) \ast H$.  Since $T'$ is a graph of actions, $T'$ is not indecomposable; on the other hand, $T'$ is mixing by Proposition \ref{P.NotMixing}.

%This example is ``almost generic'' in the following sense: Suppose that $Y \in \partial \cvn$ is mixing but not indecomposable, then $Y$ splits as a graph of actions by Lemma \ref{L.MixGraph}.  Suppose further that $Y$ is not dynamically reduced by any proper factor of $\FN$; Corollary \ref{C.CarryToReduce} gives that no finitely generated subgroup of $\FN$ carrying a leaf of $L(Y)$ is contained in a proper factor of $\FN$.  Note that $Y$ cannot be free and that further $Y$ must contain a point with non-abelian stabilizer.  

\bibliographystyle{amsplain}
\bibliography{indecompREF.bib}

\end{document}